\documentclass[12pt]{amsart}
\usepackage{hyperref}
\usepackage{amsfonts}
\usepackage{bbm}
\usepackage{esint}
\usepackage{mathrsfs}
\usepackage{amssymb,amsmath,amsfonts,amsthm}

\numberwithin{equation}{section}

\allowdisplaybreaks

\begin{document}

\baselineskip 16.1pt \hfuzz=6pt

\theoremstyle{plain}
\newtheorem{theorem}{Theorem}[section]
\newtheorem{prop}[theorem]{Proposition}
\newtheorem{lemma}[theorem]{Lemma}
\newtheorem{corollary}[theorem]{Corollary}
\newtheorem{example}[theorem]{Example}
\newtheorem*{thmA}{Theorem A}
\newtheorem*{thmB}{Theorem B}
\newtheorem*{thmC}{Theorem C}
\newtheorem*{defnA}{Definition A}
\newtheorem*{defnB}{Definition B}
\newtheorem*{defnC}{Definition C}
\newtheorem*{defnD}{Definition D}
\newtheorem*{defnE}{Definition E}

\theoremstyle{definition}
\newtheorem{definition}[theorem]{Definition}
\newtheorem{remark}[theorem]{Remark}

\renewcommand{\theequation}
{\thesection.\arabic{equation}}

\allowdisplaybreaks

\newcommand{\XX}{X}
\newcommand{\GG}{\mathop G \limits^{    \circ}}
\newcommand{\GGs}{{\mathop G\limits^{\circ}}}
\newcommand{\GGtheta}{{\mathop G\limits^{\circ}}_{\theta}}
\newcommand{\xoneandxtwo}{X_1\times\mathcal X_2}

\newcommand{\GGp}{{\mathop G\limits^{\circ}}}

\newcommand{\GGpp}{{\mathop G\limits^{\circ}}_{\theta_1,\theta_2}}

\newcommand{\e}{\varepsilon}
\newcommand{\bmo}{{\rm BMO}}
\newcommand{\vmo}{{\rm VMO}}
\newcommand{\cmo}{{\rm CMO}}
\newcommand{\Z}{\mathbb{Z}}
\newcommand{\N}{\mathbb{N}}
\newcommand{\R}{\mathbb{R}}
\newcommand{\C}{\mathbb{C}}

\pagestyle{myheadings}\markboth{\rm\small Yanchang Han, Yongsheng Han and  Ji Li
}{\rm\small Boundedness of singular integrals and Hardy space}

\title[Boundedness of singular integrals and Hardy space]
{Criterion of the boundedness of singular integrals on spaces of homogeneous type}

\author{Yanchang Han, Yongsheng Han and  Ji Li}

\thanks{The first author is supported by  NNSF of China (Grant No. 11471338) and
Guangdong province natural science foundation of China (Grant No. 2014A030313417); The third author is supported by the
Australian Research Council under Grant No.~ARC-DP120100399}



\begin{abstract}
It was well known that geometric considerations enter in a decisive way in many questions of harmonic analysis. The main purpose of this paper is to provide the criterion of the boundedness for singular integrals on the Hardy spaces and as well as on its dual, particularly on $\bmo$ for spaces of homogeneous type $(X, d,\mu)$ in the sense of Coifman and Weiss, where the quasi-metric $d$ may have no regularity and the measure $\mu$ satisfies only the doubling property. We make no additional geometric assumptions on the quasi-metric or the doubling measure and thus, the results of this paper extend to the full generality of all related previous ones, in which the extra geometric assumptions were made on both the quasi-metric $d$ and the measure $\mu.$ To achieve our goal, we prove that the atomic Hardy spaces introduced by Coifman and Weiss coincide with the Hardy spaces defined in terms of wavelet coefficients and develop the molecule theory for this general setting. The main tools used in this paper are atomic decomposition, the orthonormal wavelet basis constructed recently by Auscher and Hyt\"onen, the discrete Calder\'on-type reproducing formula, the almost orthogonal estimates, implement various stopping time arguments and the duality of the Hardy spaces with the Carleson measure spaces.
\end{abstract}
\maketitle

\tableofcontents

\section{Introduction}\label{sec:introduction}
\setcounter{equation}{0}
The classical theory of Calder\'on--Zygmund singular integral operators as well as the theory of
function spaces were based on extensive use of convolution operators and on the Fourier transform.
However, it is now possible to extend most of those ideas and results to spaces of homogeneous type.
As Meyer remarked in his
preface to~\cite{DH}, \emph{``One is amazed by the dramatic
changes that occurred in analysis during the twentieth century.
In the 1930s complex methods and Fourier series played a
seminal role. After many improvements, mostly achieved by the
Calder\'on--Zygmund school, the action takes place today on
spaces of homogeneous type. No group structure is available,
the Fourier transform is missing, but a version of harmonic
analysis is still present. Indeed the geometry is conducting
the analysis.''}

Spaces of homogeneous type were introduced by
Coifman and Weiss in the early 1970s, in~\cite{CW1}. We say
that $(X,d,\mu)$ is {\it a space of homogeneous type} in the
sense of Coifman and Weiss if $d$ is a quasi-metric on~$X$
and $\mu$ is a nonzero measure satisfying the doubling
condition. A \emph{quasi-metric}~$d$ on a set~$X$ is a
function $d:X\times X\longrightarrow[0,\infty)$ satisfying
(i) $d(x,y) = d(y,x) \geq 0$ for all $x,y\in X$; (ii)
$d(x,y) = 0$ if and only if $x = y$; and (iii) the
\emph{quasi-triangle inequality}: there is a constant $A_0\in
[1,\infty)$ such that for all $x$, $y$, $z\in X$,
\begin{eqnarray}\label{eqn:quasitriangleineq}
    d(x,y)
    \leq A_0 [d(x,z) + d(z,y)].
\end{eqnarray}
We define the quasi-metric ball by $B(x,r) := \{y\in X: d(x,y)
< r\}$ for $x\in X$ and $r > 0$. Note that the quasi-metric, in
contrast to a metric, may not be H\"older regular and
quasi-metric balls may not be open.
We say that a nonzero measure $\mu$ satisfies the
\emph{doubling condition} if there is a constant $C_\mu$ such
that for all $x\in\XX$ and $r > 0$,
\begin{eqnarray}\label{doubling condition}
  \mu(B(x,2r))
   \leq C_\mu \mu(B(x,r))
   < \infty.
\end{eqnarray}
We point out that the doubling condition (\ref{doubling
condition}) implies that there exist positive constants
$\omega$ (the \emph{upper dimension} of~$\mu$) and $C_\omega$ such
that for all $x\in X$, $\lambda\geq 1$ and $r > 0$,
\begin{eqnarray}\label{upper dimension}
    \mu(B(x, \lambda r))
    \leq C_\omega\lambda^{\omega} \mu(B(x,r)).
\end{eqnarray}
Spaces of homogeneous type include many special spaces in
analysis and have many applications in the theory of singular
integrals and function spaces; See~\cite{Chr, CW2, NS1, NS2} for more
details. Coifman and Weiss in \cite{CW2} introduced the atomic
Hardy space $H^p_{cw}(X)$ on $(X,d,\mu).$ To recall the atomic Hardy space,
one first needs the Campanato spaces $\mathcal{C}_\alpha(X)$, $\alpha \geq 0$, consisting of those
functions for which
\begin{eqnarray}\label{Lip space}\left\lbrace \frac{1}{\mu(B)}\int_B|f(x)-f_B|^2 d\mu(x)\right\rbrace^{\frac{1}{2}}
\leq C[\mu(B)]^\alpha,\end{eqnarray}
where $B$ are any quasi-metric balls, $f_B=\frac{1}{\mu(B)}\int_B f(x)d\mu(x)$ and $C$ is independent of $B.$ Let
$||f||_{{\mathcal C}_{\alpha}(X)}$ be the infimum of all $C$ for which (\ref{Lip space}) holds.
The atomic Hardy space $H^p_{cw}(X)$ introduced by Coifman and Weiss is defined to be the subspace of
the dual of $\mathcal{C}_\alpha(X)$, where $\alpha =\frac{1}{p} - 1, 0<p\leq 1$,
consisting of those linear functional
admitting an \textit{atomic decomposition}
\begin{equation}\label{atomic decomposition}
f=\sum_{j=1}^\infty \lambda_j a_j,
\end{equation}
where the $a_j^,$s are $(p, 2)$-atoms, $\sum_{j=1}^\infty |\lambda_j|^p<\infty$ and the series in (\ref{atomic decomposition}) converges in
the dual of $\mathcal{C}_\alpha(X)$, where $\alpha =\frac{1}{p} - 1, 0<p\leq 1$.

Here a function
$a(x)$ is an $(p, 2)$-atom if

(i) the support of $a(x)$ is contained in a ball $B(x_0,r)$ for $r>0$ and $x_0\in X$;

(ii)$\|a(x)\|_{L^2(X)} \leq \{\mu(B(x_0,r))\}^{\frac{1}{2}-\frac{1}{p}}$;

(iii) $\int_X a(x)d\mu(x) = 0$.

The quasi-norm of $f$ in $H^p_{cw}(X)$ is defined by
$\|f\|_{H^p_{cw}(X)}=\inf\lbrace \lbrace \sum_{j=1}^\infty |\lambda_j|^p\rbrace^{\frac{1}{p}}\rbrace,$
where the infimum is taken over all such atomic representations of $ f.$

The atomic Hardy spaces have many applications. For example, if an
operator $T$ is bounded on $L^2(X)$ and from $H^p_{cw}(X)$ to $L^p(X)$ for
some $p\leq 1,$ then $T$ is bounded on $L^q(X)$ for $1<q\leq 2.$ See \cite {CW2} for more applications.

Even though spaces of homogeneous have many applications,
however, for some applications, additional geometric assumptions were
required on the quasi-metric $d$ and the measure $\mu.$ This is because, as
mentioned, the original quasi-metric~$d$ may have no
regularity and quasi-metric balls, even Borel sets, may not be
open. For instance, to establish the maximal function
characterization for the Hardy space on spaces of homogeneous type, Mac\'ias and
Segovia in~\cite{MS1} replaced the quasi-metric~$d$ by another quasi-metric $d'$
on~$X$ such that the topologies induced on~$X$ by $d$
and~$d'$ coincide, and $d'$ has the following regularity
property:
\begin{eqnarray}\label{smooth metric}
    |d'(x,y) - d'(x',y)|
    \le C_0 \, d'(x,x')^\theta \,
        [d'(x,y) + d'(x',y)]^{1 - \theta}
\end{eqnarray}
for some constant~$C_0,$ some regularity exponent
$\theta\in(0,1)$, and for all $x$, $x'$, $y\in X$. Moreover, if
quasi-metric balls are defined by this new quasi-metric $d'$,
that is, $B'(x,r) := \{y\in X: d'(x,y) < r\}$ for $r > 0$, then
the measure $\mu$ satisfies the following property:
\begin{eqnarray}\label{regular}
    \mu(B'(x,r))\sim r.
\end{eqnarray}
Note that property~\eqref{regular} is much stronger than the
doubling condition. Mac\'{i}as and Segovia \cite{MS1} first introduced test function and
distribution spaces based on
the conditions \eqref{smooth metric} and ~\eqref{regular}, and then established the
maximal function characterization for Hardy spaces $H^p_{max}(X)$
with $(1 + \theta)^{-1} < p \leq 1$, on spaces of homogeneous
type~$(X,d',\mu)$ that satisfy the regularity
condition~\eqref{smooth metric} on the quasi-metric~$d'$ and
property~\eqref{regular} on the measure~$\mu$. The most remarkable work on $(X,d',\mu)$ is
the $Tb$ theorem of David, Journ\'e and Semmes \cite{DJS}.
See also \cite{DH} and \cite{HS} for the Littlewood--Paley square function characterization
of the Hardy, Besove and Triebel--Lizorkin spaces on such spaces $(X,d',\mu)$.

This theme has now been developed systematically by a number of people. In \cite{NS1}, Nagel and Stein developed the product $L^p$ $(1
< p < \infty)$ theory in the setting of the
Carnot--Carath\'eodory spaces formed by vector fields
satisfying H\"{o}rmander's finite rank condition. The
Carnot--Carath\'eodory spaces studied in~\cite{NS1} are spaces
of homogeneous type with a smooth quasi-metric $d$ and a
measure~$\mu$ satisfying the conditions $\mu(B(x, sr)) \sim
s^{m+2} \mu(B(x,r))$ for $s\geq 1$ and $\mu(B(x, sr)) \sim
s^4\mu(B(x,r))$ for $s\leq 1.$ These conditions on the measure
are weaker than property~ in (\ref{regular}) but are still stronger
than the original doubling condition. In~\cite{HMY}, motivated by the work of Nagel and
Stein, Hardy spaces, namely the atomic and the Littlewood-Paley
square function characterizations, were developed on spaces of homogeneous
type $(X, d, \mu)$ with the quasi-metric $d$ satisfies the regular property in (\ref{smooth metric}) and
the measure $\mu$ satisfies the
above conditions which are stronger than the doubling property in (\ref{doubling
condition}).

More recently, Auscher and Hyt\"onen~\cite{AH} constructed an orthonormal
wavelet basis with H\"older regularity and exponential decay
for spaces of homogeneous type in the sense of Coifman and Weiss. This result is
remarkable since there are no additional geometric assumptions other than
those defining spaces of homogeneous type. To be precise, Auscher and Hyt\"onen proved the following

\begin{thmA}[\cite{AH} Theorem 7.1]\label{theorem AH orth basis}
    Let $(X,d,\mu)$ be a space of homogeneous type in the sense of Coifman and Weiss with
    quasi-triangle constant $A_0.$ There exists an orthonormal wavelet basis
    $\{\psi_\alpha^k\}$, $k\in\mathbb{Z}$, $x_\alpha^k\in
    \mathscr{Y}^k$, of $L^2(X)$, having exponential decay
    \begin{eqnarray}\label{exponential decay}
        |\psi_\alpha^k(x)|
        \leq {C\over \sqrt{\mu(B(x_\alpha^k,\delta^k))}}
            \exp\Big(-\nu\Big( {d(x^k_\alpha,x)\over\delta^k}\Big)^a\Big),
    \end{eqnarray}
    H\"older regularity
    \begin{eqnarray}\label{Holder-regularity}
        |\psi_\alpha^k(x)-\psi_\alpha^k(y)|
        \leq \frac{C}{\sqrt{\mu(B(x_\alpha^k,\delta^k))}}
            \Big( \frac{d(x,y)}{\delta^k}\Big)^\eta
            \exp\Big(-\nu\Big( \frac{d(x^k_\alpha,x)}{\delta^k}\Big)^a\Big)
    \end{eqnarray}
    for $d(x,y)\leq \delta^k$, and the cancellation property
    \begin{eqnarray}\label{cancellation}
        \int_X \psi_\alpha^k(x)\,d\mu(x) = 0,
        \qquad \text{for }k\in\mathbb{Z}.
    \end{eqnarray}
Moreover, the wavelet expansion is given by
\begin{eqnarray}\label{eqn:AH_reproducing formula}
    f(x)
    = \sum_{k\in\mathbb{Z}}\sum_{\alpha \in \mathscr{Y}^k}
        \langle f,\psi_{\alpha}^k \rangle \psi_{\alpha}^k(x)
\end{eqnarray}
in the sense of $L^2(X)$.
\end{thmA}
Here $\delta$ is a fixed small parameter, say $\delta \leq
10^{-3} A_0^{-10}$, $a= (1+2\log_2A_0)^{-1},$ and $C < \infty$, $\nu > 0$ and
$\eta\in(0,1]$ are constants independent of $k$, $\alpha$, $x$
and~$x_\alpha^k$. See \cite{AH} for more notations and details of the proof.

Auscher and Hyt\"onen's orthonormal wavelet bases open the door
for developing wavelet analysis on spaces of homogeneous type. For example, applying
orthonormal wavelet bases, Auscher and Hyt\"onen \cite{AH} proved the $T(1)$ theorem on spaces of homogeneous type
in the sense of Coifman and Weiss. Motivated by Auscher and
Hyt\"onen's work, in \cite{HLW} the Hardy space theory was developed
on space of homogeneous type in the sense of
Coifman and Weiss. More precisely, let $(X,d,\mu)$ be space of homogeneous type in the sense of Coifman and Weiss with $\mu(X)=\infty.$ They first introduce the test function and distribution space as follows.

\begin{defnA}[\cite{HLW}]\label{def-of-test-func-space}
    \textup{(Test functions)} Fix $x_0\in X$, $r > 0$,
    $\beta\in(0,\eta]$ where $\eta$ is the regularity exponent
    from Theorem A and $\gamma
    > 0$. A function $f$ defined on~$X$ is said to be a {\it test
    function of type $(x_0,r,\beta,\gamma)$ centered at $x_0\in
    X$} if $f$ satisfies the following three conditions.
    \begin{enumerate}
        \item[(i)] \textup{(Size condition)} For all $x\in
            X$,
            \[
                |f(x)|
                \leq C \,\frac{1}{V_{r}(x_0) + V(x,x_0)}
                \Big(\frac{r}{r + d(x,x_0)}\Big)^{\gamma}.
            \]

        \item[(ii)] \textup{(H\"older regularity
            condition)} For all $x$, $y\in X$ with $d(x,y)
            < (2A_0)^{-1}(r + d(x,x_0))$,
            \[
                |f(x) - f(y)|
                \leq C \Big(\frac{d(x,y)}{r + d(x,x_0)}\Big)^{\beta}
                \frac{1}{V_{r}(x_0) + V(x,x_0)} \,
                \Big(\frac{r}{r + d(x,x_0)}\Big)^{\gamma}.
            \]

        \item[(iii)] \textup{(Cancellation condition)}
            \[
                \int_X f(x) \, d\mu(x)
                = 0,
            \]
    \end{enumerate}
where $V_r(x_0)=\mu(B(x_0,r))$ and $V(x_0,x)=\mu(B(x_0, d(x,x_0))).$
\end{defnA}
Note that, as proved in \cite{HLW}, $\psi_\alpha^k(x)\over \sqrt{\mu(B(x_\alpha^k,\delta^k))}$ is a test function with $x_0=x_\alpha^k, r=\delta^k, \beta=\eta$ and any $\gamma>0.$ The test function space is denoted by $G(x_0,r,\beta,\gamma),$ which consists of all test functions of type $(x_0,r,\beta,\gamma)$. The norm of $f$ in $
G(x_0,r,\beta,\gamma)$ is defined by
\[
    \|f\|_{G(x_0,r,\beta,\gamma)}
    := \inf\{C>0:\ {\rm(i)\  and \ (ii)}\ {\rm hold} \}.
\]

For each fixed $x_0$, let $G(\beta,\gamma) :=
G(x_0,1,\beta,\gamma)$. It is easy to check that for each fixed
$x_0'\in X$ and $r > 0$, we have $G(x_0',r,\beta,\gamma) =
G(\beta,\gamma)$ with equivalent norms. Furthermore, it is also
easy to see that $G(\beta,\gamma)$ is a Banach space with
respect to the norm on $G(\beta,\gamma)$.

For $\beta \in (0,\eta]$ and $\gamma > 0$, let
$\GGs(\beta,\gamma)$ be the completion of the space
$G(\eta,\gamma)$ in the norm of $G(\beta,\gamma)$; of course
when $\beta = \eta$ we simply have $\GGs(\beta,\gamma) =
\GGs(\eta,\gamma) = G(\eta,\gamma)$. We define the norm on
$\GGs(\beta,\gamma)$ by $\|f\|_{\GGs(\beta,\gamma)} :=
\|f\|_{G(\beta,\gamma)}$.

\begin{defnB}[\cite{HLW}]
    \textup{(Distributions)} Fix $x_0\in X$, $r > 0$,
    $\beta\in(0,\eta]$ where $\eta$ is the regularity exponent
    from Theorem A and $\gamma > 0$.
    The \emph{distribution space} $(\GGs(\beta,\gamma))'$ is
    defined to be the set of all linear functionals
    $\mathcal{L}$ from $\GGs(\beta,\gamma)$ to $\mathbb{C}$
    with the property that there exists $C > 0$ such that for
    all $f\in \GGs(\beta,\gamma)$,
    \[
        |\mathcal{L}(f)|
        \leq C\|f\|_{\GGs(\beta,\gamma)}.
    \]
\end{defnB}
A fundamental result proved in \cite{HLW} is the following wavelet representation
for test functions and distributions.
\begin{thmB}[\cite{HLW}]\label{thm reproducing formula test function}
    \textup{(Wavelet reproducing formula for test functions and distributions)} Suppose
    that $f\in \GGs(\beta,\gamma)$ with $\beta$,
    $\gamma\in(0,\eta)$. Then the wavelet reproducing formula
    \begin{equation}\label{eqn:reproducing_formula}
    f(x)
        = \sum_{k\in\mathbb{Z}}\sum_{\alpha \in \mathscr{Y}^k}
            \langle f,\psi_{\alpha}^k \rangle \psi_{\alpha}^k(x)
    \end{equation}
    holds in $\GGs(\beta',\gamma')$ for all $\beta'\in(0,\beta)$ and
    $\gamma'\in(0,\gamma)$. Moreover, the wavelet
    reproducing formula~\eqref{eqn:reproducing_formula} also
    holds in the space $(\GGs(\beta,\gamma))'$ of
    distributions.
    \end{thmB}

Based on the above wavelet reproducing formula, the Littlewood--Paley square function
in terms of wavelet coefficients is defined by the following
\begin{defnC}[\cite {HLW}]\label{def:discrete_square_function}
For $f$ in $(\GGs(\beta,\gamma))'$ with $\beta$, $\gamma\in(0,\eta)$, the \emph{discrete
Littlewood--Paley square function $S(f)$ of $f$} is defined by

    \begin{eqnarray}
        S(f)(x)
        := \Big\{ \sum_{k}\sum_{\alpha\in\mathscr{Y}^k} \big|
            \langle \psi_{\alpha}^{k},f \rangle
            \widetilde{\chi}_{{Q}_{\alpha}^{k}}(x)
            \big|^2 \Big\}^{1/2},
    \end{eqnarray}
    where $\widetilde{\chi}_{{ Q}_{\alpha}^{k}}(x) := \chi_{{
    Q}_{\alpha}^{k}}(x) \mu({ Q}_{\alpha}^{k})^{-1/2}$ and $\chi_{{
    Q}_{\alpha}^{k}}(x)$ is the indicator function of the dyadic
    cube~${ Q}_{\alpha}^{k}$.
\end{defnC}
The Hardy space on space of homogeneous type in the sense of
Coifman and Weiss then is introduced as follows.
\begin{defnD}[\cite{HLW}] \label{def-Hp}
    Suppose that $0<\beta, \gamma<\eta$ and $\frac{\omega}{\omega+\eta} < p \leq 1$, where
    $\eta$ is the regularity given in Theorem A and $\omega$ is the upper dimension of $X.$
    The {\it Hardy space} $H^p(X)$ is defined by
    \[
        H^p(X)
        :=\big\lbrace f \in (\GG(\beta,\gamma))':
            S(f)\in L^p(X)\big\rbrace.
    \]
    The norm of $f\in H^p(X)$ is defined by
    $\|f\|_{H^p(X)} := \|S(f)\|_{L^p(X)}$.
\end{defnD}

Natural questions arise:

(1) Is the atomic Hardy space $H^p_{cw}(X)$ same as the Hardy space $H^p(X)$ with equivalent norms?

(2) Can one provide a criterion of the boundedness for singular integral operators on these Hardy spaces?

In this paper, we address the above questions. We will give a positive answer for the first question by the following
\begin{theorem}\label{thm 1.1}
Let $(X,d,\mu)$ be space of homogeneous type in the sense of Coifman and Weiss and $\frac{\omega}{\omega+\eta}< p \leq 1.$ Then
$H^p_{cw}(X)=H^p(X)$ with equivalent norms. More precisely, if $f\in H^p_{cw}(X)$ then $f\in H^p(X)$ and there exists a constant $C$ such that $\|f\|_{H^p}\leq C\|f\|_{H^p_{cw}}.$ Conversely, if $f\in H^p(X)$ then $f$ has an atomic decomposition:
$f=\sum_j \lambda_j a_j$ where all $a_j's$ are $(p,2)$ atoms and the series converges in the dual of $\mathcal{C}_\alpha(X), \alpha =\frac{1}{p} - 1.$ Moreover, $\sum_j |\lambda_j|^p\leq C\|f\|_{H^p}^p$
where the constant $C$ is independent of $f.$
\end{theorem}

We would like to point out that the most significate integrant of Theorem \ref{thm 1.1} is the method of atomic decomposition for subspace $L^2(X)\cap H^p(X)$. More precisely, if $f\in L^2(X)\cap H^p(X)$ then $f$ has an atomic decomposition which converges in both $L^2(X)$ and $H^p(X).$ These facts play a crucial role in this paper. We also remark that if $f$ is a distribution in $(\GGs(\beta,\gamma))'$, in general, $f$ may not be a linear functional on $\mathcal{C}_\alpha(X).$ However, Theorem \ref{thm 1.1} implies that if $f$ is a distribution in $(\GGs(\beta,\gamma))'$ and belongs to $H^p(X),$ then $f$ can be defined as a linear functional on $\mathcal{C}_\alpha(X).$ One may also observe that the wavelet reproducing formula is not available for providing an atomic decomposition for $H^p(X)$ since the wavelets $\psi^k_\alpha(x)$ have no compact supports. To overcome this problem, a crucial idea is to establish a new kind of
Calder\'on-type reproducing formula. See Proposition \ref{ new Caderon-type reproducing formula} below for such a new reproducing formula. As mentioned before, Mac\'{i}as and Segovia \cite{MS2} gave the maximal function characterization of the Hardy space only for the so-called normal spaces $(X, d,\mu)$ where $d$ satisfies the regularity condition in \eqref{smooth metric}
 and $\mu$ satisfies the condition in \eqref{regular}. They proved the relation between the atomic Hardy space $H^p_{cw}(X)$ and the maximal Hardy space $H^p_{max}(X)$ only in the sense that if $f$ in $H^p_{cw}(X),$ denoting by $\widetilde f$ the restriction of $f$ to $E^\alpha,$ the test function space on normal space $(X,d,\mu),$ then $\mathcal{F}f=\widetilde f$ defines an injective linear transformation from $H^p_{cw}(X)$ onto the space of the distribution $g$ on $E^\alpha$ such that $g^*_\gamma(x)$ belongs to $L^p(X, d\mu))$ and there exist two positive and finite constants $c_1$ and $c_2$ such that
$$c_1\|f\|_{H^p_{cw}}\leq \Big(\int{\widetilde f}^*_\gamma(x)^p d\mu(x)\Big)^{\frac{1}{p}}\leq c_2\|f\|_{H^p_{cw}}.$$
See Theorem 5.9 in \cite{MS2} for more details.

In order to provide the criterion of the boundedness for singular integrals on the Hardy spaces, we now define
singular integral operator on spaces of homogeneous type in the sense of Coifman and Weiss.
\begin{definition}\label{def 1}
We say that $T$ is a singular integral operator on space of homogeneous type $(X,d,\mu)$ if $T$ is of the form
$T(f)(x)=\int K(x,y)f(y) d\mu(y),$ where $K(x,y),$ the kernel of $T,$ satisfies the following estimates:

\begin{eqnarray}\label{size of C-Z-S-I-O}
    |K(x, y)|
    \leq {{C}\over {V(x, y)}}
\end{eqnarray}
for all $x\not= y$;
\begin{eqnarray}\label{x smooth of C-Z-S-I-O}
    |  K(x, y) - K(x', y) |
    \leq {C\over V(x,y)}\Big({d(x, x')\over d(x,y)}\Big)^\eta
\end{eqnarray}
for $d(x, x')\leq (2A_0)^{-1} d(x, y)$;
\begin{eqnarray}\label{y smooth of C-Z-S-I-O}
 |  K(x, y) - K(x, y') |
    \leq {C\over V(x,y)} \Big({d(y, y')\over d(x,y)}\Big)^\eta
\end{eqnarray}
for $d(y, y')\leq (2A_0)^{-1} d(x, y).$
\end{definition}
The criterion of the boundedness for singular integrals on the Hardy space $H^p(X),$ i.e., the answer of the second question, is the following
\begin{theorem}\label{thm T1Hp}
  Suppose that $T$ is a singular integral operator with the kernel $K(x,y)$ satisfying the estimates \eqref{size of C-Z-S-I-O} and \eqref{y smooth of C-Z-S-I-O}, and $T$
  is bounded on $L^2(X).$ Then $T$ extends to be a bounded operator on $H^p(X), \frac{\omega}{\omega+\eta}< p \leq 1,$ if and only if $T^*(1)=0$.
\end{theorem}
Here, if $T$ is bounded on $L^2(X)$ and $H^p(X),$ then $T^*(1)=0$ means that
$$\langle f, T^*(1)\rangle=\langle T(f), 1\rangle=\int T(f)(x) d\mu(x)=0$$
for all $f\in L^2(X)\cap H^p(X), \frac{\omega}{\omega+\eta}< p \leq 1.$

We would like to remark that by Theorem \ref{thm 1.1}, one could state Theorem \ref{thm T1Hp} with $H^p(X)$ replaced by $H^p_{cw}(X).$
The reason for not doing this is that the atomic Hardy space $H^p_{cw}(X)$ is not convenient
for proving the boundedness of operators. Indeed, if $f\in H^p_{cw}(X)$ has an atomic decomposition
$f(x)=\sum_j \lambda_j a_j(x),$ in general, $T(f)(x)$ can not be written as $\sum_j \lambda_j T(a_j)(x)$
even when $T$ is bounded on $L^2(X)$ and $f\in L^2(X)\cap H^p_{cw}(X).$ However, $H^p(X)$ is more convenient to use for proving the boundedness of
operators on the Hardy space. This is because $L^2(X)\cap H^p(X)$ is dense in $H^p(X)$ and, as proved in Theorem \ref{thm 1.1}, $f\in L^2(X)\cap H^p(X)$ has a nice atomic decomposition which converges
in both $L^2(X)$ and $H^p(X).$ Therefore, if $T$ is bounded on $L^2(X)$ one can get
$T(f)(x)=\sum_j \lambda_j T(a_j)(x).$ Then it suffices to verify that $T(a)$ is in $H^p(X)$ with the upper
bound uniformly for all $(p,2)$ atoms $a(x).$ And this can be concluded by applying
the molecule theory. Note that the molecule theory was developed by Coifman and Weiss for $(X,\rho,\mu)$ where $\rho$ is the measure distance, see page 594 in \cite{CW2}. In this paper, we develop the molecule theory for $(X,d,\mu)$ with the original quasi metric $d$ and the doubling measure $\mu,$ see Theorem \ref{moleculeinHp} below. Moreover, the method of atomic decomposition for subspace $L^2(X)\cap H^p(X)$ will also be applied for the proof of
the necessary condition that $T^*(1)=0$ if $T$ is bounded on $L^2(X)$ and $H^p(X).$ Note that this necessary condition on
$ \mathbb{R}^n$ was obtained directly from the fact that if $f\in L^2(\mathbb{R}^n)\cap H^p(\mathbb{R}^n)$ then $\int_{ \mathbb{R}^n} f(x)dx=0.$ This last fact follows from the estimate of the Fourier transform for $f\in L^2(\mathbb{R}^n)\cap H^p(\mathbb{R}^n)$.
Since the Fourier transform is missing on general spaces of homogeneous type, to show $\int_X f(x) d\mu(x)=0$ for all $f\in L^2(X)\cap H^p(X), \frac{\omega}{\omega+\eta}< p \leq 1,$ a new approach used in this paper is to prove the estimate
$$||f||_{L^p}\leq C||f||_{H^p}$$
for $f\in L^2(X)\cap H^p(X)$ with the constant $C$ independent of the $L^2(X)$ norm of $f.$ This estimate has their own interest and the method of the atomic decomposition of subspace $L^2(X)\cap H^p(X)$ plays a crucial role in the proof for such an estimate. See Proposition \ref{HpinLp} below for details of the proof.

The last main result in this paper is the boundedness of singular integrals on the dual of the Hardy space. It was
well-known that the Campanato space $\mathcal{C}_{\frac{1}{p} - 1}(X), 0<p\leq 1,$ is the dual of the atomic
Hardy space $H^p_{cw}(X)$ as well as that $\bmo(X)$ is the dual of $H^1_{cw}(X).$  In \cite{HLW} the Carleson measure spaces $\cmo^p(X)$ were introduced and it was proved that space $\cmo^p(X)$ is the dual of $H^p(X)$ as well as $\cmo^1(X)=\bmo(X)$ is the dual of $H^1(X).$ We will prove the boundedness of singular integrals on $\cmo^p(X).$ The reason for doing this is that we will show that $L^2(X)\cap \cmo^p(X)$ is a dense subspace of $\cmo^p(X)$ in the weak topology sense. See Lemma \ref{weaktopology} below. This weak density argument plays a similar role as the subspace $L^2(X)\cap H^p(X)$ does in the proofs of Theorem \ref{thm 1.1} and Theorem \ref{thm T1Hp}. To be precise, the space $\cmo^p(X)$ is defined by the following

\begin{defnE}[\cite{HLW}]\label{def-CMO}
    Let $(X,d,\mu)$ be space of homogeneous type in the sense of Coifman and Weiss. Suppose that ${\omega\over
\omega+\eta}  < p \leq 1$, where
    $\omega$ is the upper dimension of $X$.
    The {\it Carleson measure space} $\cmo^p(X)$ is defined by
$$
        \cmo^p(X)
        := \big\{ f \in (\GG(\beta, \gamma))' : \mathcal{C}_p(f)< \infty\},
$$
    where
    $$\mathcal{C}_p(f):= \sup_{Q}\Big\{ \frac{1}{\mu(Q)^{{2\over p} - 1}}
            \sum_{k\in\mathbb{Z}, \alpha\in\mathscr{Y}^{k},
                 Q_{\alpha}^{k} \subset Q}
            \big| \langle \psi_{\alpha}^{k}, f \rangle
            \big|^2 \Big\}^{1/2},$$
    where $Q$ runs over all quasi-metric dyadic balls in the sense of Auscher and Hyt\"onen.
\end{defnE}
In \cite{HLW}, the following duality between $H^p(X)$ and $\cmo^p(X)$ was proved.
\begin{thmC}[\cite{HLW}]\label{thm-duality}
    Suppose ${\omega\over
\omega+\eta}  < p \leq 1$, where
    $\omega$ is the upper dimension of $X$.
    Then the Carleson measure space~$\cmo^p(X)$ is the
    dual of the Hardy space~$H^p(X)$:
    \[
        \big(H^p(X)\big)'
        = \cmo^p(X).
    \]
    More precisely, if $g\in \cmo^p(X)$ the map ${\ell }_g$ given by
${\ell}_g(f)=\langle f, g\rangle,$ defined initially for $f\in \GG(\beta, \gamma),$ extends to a continuous linear functional on $H^p(X)$ with $\Vert \ell_g\Vert \approx \Vert g\Vert_{\cmo^p(X)}.$ Conversely, for every ${\ell}\in (H^p(X)^*$ there exists some $g\in \cmo^p(X)$ so that ${\ell}={\ell}_g.$

    In particular,
    \[
        \big(H^1(X)\big)'
        = \bmo(X)=\cmo^1(X).
    \]
\end{thmC}
The last result in this paper is the following
\begin{theorem}\label{thm T1CMOP} If $T$ is a singular integral with the kernel $K(x,y)$ satisfying the estimates \eqref{size of C-Z-S-I-O} and \eqref{x smooth of C-Z-S-I-O}, and $T$ is bounded on $L^2(X)$ then $T$
extends to be a bounded operaor on $\cmo^p(X)$ if and only if $T(1)=0$.
\end{theorem}
Here, again, if $T$ is bounded on $L^2(X)$ and $\cmo^p(X), T(1)=0$ means that $\langle T1, f\rangle=\langle 1, T^*(f)\rangle=\int_X T^*(f)(x) d\mu(x)=0$ for all $f\in L^2(X)\cap H^p(X), \frac{\omega}{\omega+\eta}< p \leq 1.$

Finally, we will show that $\cmo^p(X)=\mathcal C_{\frac{1}{p}-1}(X), \frac{\omega}{\omega+\eta}<p\leq 1,$ with the equivalent norms. See Proposition \ref{cmopC} below. Hence, by the above theorem, we also provide the boundedness of singular integrals on the Campanato space.

The paper is organized as follows. In Section 2, we establish a new reproducing formula (Proposition \ref{ new Caderon-type reproducing formula}) and then prove Theorem \ref{thm 1.1}.
In Section 3, we develop the molecule theory for Hardy spaces $H^p(X)$ (Theorem \ref{moleculeinHp}) and prove Theorem \ref{thm T1Hp}. In the last section we show the weak density argument (Lemma \ref{weaktopology}) and give the proof of Theorem \ref{thm T1CMOP}.

\section{Equivalence of $H^p_{cw}(X)$ and $H^p(X)$}\label{sec:proof of theorem 1.7}
\setcounter{equation}{0}
\subsection{The proof for $\|f\|_{H^p}\lesssim \|f\|_{H^p_{cw}}$}

To show $\|f\|_{H^p}\lesssim \|f\|_{H^p_{cw}},$ we first need the following
\begin{lemma} \label{lemma:Lip(p,2)}
 Suppose that $f\in \GGs(\beta,\gamma)$ with $0<\beta\leq \eta,\frac{\omega}{\omega+\eta} <p \leq 1,\gamma > \omega(1/p-1).$ Then $f \in \mathcal C_{\frac{1}{p}-1}(X)$. Particularly,
the wavelet basis $\psi_\alpha^k(x)$ belongs to $\mathcal C_{\frac{1}{p}-1}(X)$.
\end{lemma}
\begin{proof} Suppose that $f\in \GGs(\beta,\gamma)$ with $\|f\|_{\GGs(\beta,\gamma)}=1$ and  $0<\beta\leq\eta,\frac{\omega}{\omega+\eta} <p \leq 1,\gamma > \omega(1/p-1).$ Let $B=B(x_B, r)$ with $x_B\in X, r>0$ be any fixed quasi-ball. To show that $f$ belongs to $\mathcal C_{\frac{1}{p}-1}(X),$ we consider that $r$ is large and it is small, where the size and the smoothness conditions on $f$ will be applied, respectively. To be more precise, for large $r,$ this means that $r \geq \frac{1}{4A^2_0},$ we consider two cases: $d(x_B,x_0) \leq 2A_0r$ and $d(x_B,x_0) > 2A_0r.$ For the first case, if $y\in B(x_0,1)$ then $d(x_B, y)\leq A_0(1+d(x_B,x_0))\leq A_0(2A_0r+1),$  by the doubling property on $\mu,$ thus $V(x_0,1) \leq V(x_B, A_0(2A_0r+1)) \leq C\big(\frac{A_0(2A_0r+1)}{r}\big)^{\omega}V(x_B, r)\leq CV(x_B, r)=C\mu(B)$ and by the size condition on $f,$ we have
$$
\left ( \frac{1}{\mu(B)} \int_B \left | f(x)-f_B \right |^2 d\mu(x) \right )^{1/2} \leq C \frac{\left \| f \right \|_{\GGs(\beta,\gamma)}}{V_1(x_0)}\leq C (\mu(B))^{1/p-1} \big(V_1(x_0)\big)^{-\frac{1}{p}}.
$$
For the second case, that is, $r\geq \frac{1}{4A^2_0}$ and $d(x_B,x_0) > 2A_0r.$ If $x\in B(x_B,r),$
by the quasi-inequality, $d(x_B,x_0)\leq A_0[d(x_B,x)+d(x,x_0)]\leq A_0[r+d(x,x_0)].$ This together with the fact that $r<\frac{d(x_B,x_0)}{2A_0}$ implies that $d(x,x_0)\geq \frac{1}{2A_0}d(x_B,x_0).$ Similarly, if $y\in B(x_B,r)$ then $d(y,x_0)\geq \frac{1}{2A_0}d(x_B,x_0).$ Therefore, for the second case, by the size condition on $f$ and for all $x,y\in B(x_B,r),$
\begin{eqnarray*}\left | f(x)-f(y) \right |
\leq C \frac{1}{V_1(x_0)} \left ( \frac{1}{1+d(x_B,x_0)} \right )^{\gamma}\left \| f \right \|_{\GGs(\beta,\gamma)}.
\end{eqnarray*}
Note that $B(x_0,1)\subset B\big(x_B, A_0(1+d(x_B,x_0)\big)$ and hence, the doubling property on $\mu$ implies that
$V(x_0,1) \leq V(x_B, A_0(1+d(x_B,x_0)))\leq C \big( \frac{A_0(1+d(x_B,x_0))}{r}\big)^\omega V(x_B, r).$ We obtain that
\begin{eqnarray*}
&&\left ( \frac{1}{\mu(B)} \int_B \left | f(x)-f_B \right |^2 d\mu(x) \right )^{1/2}
\\&\leq& C \frac{1}{V_1(x_0)} \left ( \frac{1}{1+d(x_B,x_0)} \right )^\gamma \left \| f \right \|_{\GGs(\beta,\gamma)}\\
&\leq& C \frac{1}{V_1(x_0)} \left ( \frac{1}{1+d(x_B,x_0)} \right )^\gamma\bigg ( {\frac{A_0(1+d(x_B,x_0))}{r}}\bigg)^{\omega(\frac{1}{p}-1)}
\bigg(\frac{V(x_B,r)}{V(x_0,1)}\bigg)^{\frac{1}{p}-1}
\\&\leq& C (\mu(B))^{1/p-1} \big(V_1(x_0)\big)^{-\frac{1}{p}}
\end{eqnarray*}
since $r>\frac{1}{4A^2_0}$ and $\gamma > \omega(1/p-1).$

We now consider the small $r,$ that is, $r<\frac{1}{4A^2_0}.$ Note that in this case, if $x,y\in B(x_B, r)$ then $d(x,y)\leq 2A_0r\leq \frac{1}{2A_0}.$ Thus, we can apply the smoothness condition on $f$ to get
$$|f(x)-f(y)|\leq C\Big(\frac{2A_0r}{1+d(x,x_0)}\Big)^\eta \frac{1}{V_1(x_0)}\Big(\frac{1}{1+d(x,x_0)}\Big)^\gamma\left \| f \right \|_{\GGs(\beta,\gamma)}.$$
Similarly, we consider two cases: $d(x_B,x_0) \leq 2A_0r$ and $d(x_B,x_0) > 2A_0r.$ The same conclusions hold, that is, $V(x_0,1) \leq C\big(\frac{A_0(2A_0r+1)}{r}\big)^{\omega}V(x_B, r)$ and $d(x,x_0)\geq \frac{1}{2A_0}d(x_B,x_0), V(x_0,1) \leq C \big( \frac{A_0(1+d(x_B,x_0))}{r}\big)^\omega V(x_B, r),$ respectively for these two cases. Therefore, for the first case, we obtain that
\[
\begin{split}
\left ( \frac{1}{\mu(B)} \int_B \left | f(x)-f_B \right |^2 d\mu(x) \right )^{1/2} &\leq C \frac{\left \| f \right \|_{\GGs(\beta,\gamma)}}{V_1(x_0)}r^\eta\\
&\leq C r^\eta\Big(\frac{A_0(2A_0r+1)}{r}\Big)^{\omega(\frac{1}{p}-1)} (\mu(B))^{1/p-1} \big(V_1(x_0)\big)^{-\frac{1}{p}}\\
&\leq C(\mu(B))^{1/p-1} \big(V_1(x_0)\big)^{-\frac{1}{p}}
\end{split}
\]
since $r<\frac{1}{4A^2_0}$ and the condition $\frac{\omega}{\omega+\eta} <p \leq 1$ implies $\eta>\omega(\frac{1}{p}-1).$

While for the second case, we have
\begin{eqnarray*}
&&\left ( \frac{1}{\mu(B)} \int_B \left | f(x)-f_B \right |^2 d\mu(x) \right )^{1/2} \leq C \frac{\left \| f \right \|_{\GGs(\beta,\gamma)}}{V_1(x_0)}\Big(\frac{r}{1+d(x_B,x_0)}\Big)^\eta\Big(\frac{1}{1+d(x_B,x_0)}\Big)^\gamma\\
&\leq& C \Big(\frac{r}{1+d(x_B,x_0)}\Big)^\eta \Big(\frac{1}{1+d(x_B,x_0)}\Big)^\gamma\Big( \frac{A_0(1+d(x_B,x_0))}{r}\Big)^{\omega(\frac{1}{p}-1)} (\mu(B))^{1/p-1} \big(V_1(x_0)\big)^{-\frac{1}{p}}\\
&\leq& C(\mu(B))^{1/p-1} \big(V_1(x_0)\big)^{-\frac{1}{p}}
\end{eqnarray*}
since $\eta>\omega(\frac{1}{p}-1)$ and $\gamma>\omega(\frac{1}{p}-1).$

By a result in \cite{HLW}, $\psi_\alpha^k(x)/
\sqrt{\mu(B(y_\alpha^k,\delta^k))}$ belongs to $\GGs(\eta,\gamma)$ with any $\gamma>0$ and hence $\psi_\alpha^k(x)$ belongs to $\mathcal C_{\frac{1}{p}-1}(X)$. The proof of Lemma \ref{lemma:Lip(p,2)} is complete.
\end{proof}

We now prove $||f||_{H^p}\lesssim \|f\|_{H^p_{cw}}.$ Suppose that $f\in H^p_{cw}(X)$ and $f=\sum_j \lambda_j a_j$ where $a_j$ are $(p, 2)$ atoms, $\sum_j |\lambda_j|^p<\infty,$ and the series converges in $(\mathcal C_{\frac{1}{p}-1}(X))'.$ Thus, by the above lemma, $S(f)\leq \sum_j |\lambda_j||S(a_j)|$ and hence, $\|f\|^p_{H^p}=\|S(f)\|^p_{L^p}\leq \sum_j |\lambda_j|^p \|S(a_j)\|^p_{L^p}.$ We claim that for each $(p,2)$ atom $a(x),$
\begin{eqnarray}
\|S(a)\|_{L^p}\leq C,
\end{eqnarray}\label{def:claim}
where the constant $C$ is independent of $a(x).$ The claim then implies that $\|f\|_{H^p}\leq C (\sum_j |\lambda_j|^p)^{\frac{1}{p}}.$ Taking the infimum for all representations of $f$ gives the desired result. To verify  the claim \eqref{def:claim} and simplify the calculation, we will apply a result proved in \cite{HLW}. More precisely, we need the following
\begin{definition}\label{def:continuous_square_function}
    \textup{(Continuous square function)} Let $D_k(x,y) =
    \sum_{\alpha \in \mathscr{Y}^k}\psi_{\alpha}^k(x)
    \psi_{\alpha}^k(y)$. For $f\in (\GGs(\beta,\gamma))'$ with
    $\beta$, $\gamma\in(0,\eta)$, the {\it continuous
    Littlewood--Paley square function $S_c(f)$ of $f$} is
    defined by
    \[
        S_c(f)(x)
        := \Big\{ \sum_{k}|D_k(f)(x)|^2 \Big\}^{1/2}.
    \]
\end{definition}
The following result was proved in \cite{HLW}.
\begin{theorem}\label{theorem Littlewood Paley}
    \textup{(Littlewood--Paley theory)} Fix $\beta$, $\gamma\in
    (0,\eta)$ and $p\in \big(\frac{\omega}{\omega +
    \eta},\infty\big)$, where $\omega$ is the upper dimension
    of~$(X,d,\mu)$. For $f$ in $L^2(X)$, we have
    $$ \|S(f)\|_{L^p} \sim \|S_c(f)\|_{L^p}. $$
   \end{theorem}
By Theorem \ref{theorem Littlewood Paley}, to verify the claim in \eqref{def:claim}, it suffices to show
\begin{eqnarray}
\|S_c(a)\|_{L^p}\leq C
\end{eqnarray}\label{def:claim1}
for any $(p,2)$ atom $a$ and the constant $C$ independent of atoms $a.$ To do this, suppose that $a$ is an $(p,2)$ atom supported in the ball $B(x_0,r)$. Thus,
\[
\begin{split}
\int_X S_c(a)^p(x)d\mu(x) &= \int _{B(x_0,2A_0r)}S_c(a)^p(x)d\mu(x)+\int_{B(x_0,2A_0r)^c}S_c(a)^p(x)d\mu(x)\\
&=I+II.
\end{split}
\]
Applying H\"{o}lder inequality, the $L^2$ boundedness of $S_c(f),$ the size condition on $a$ and the doubling property on the measure $\mu$ imply that
\[
\begin{split}
I &\leq C \mu \left(B(x_0,2A_0r) \right)^{1-\frac{p}{2}} \left(\left \| S_c(a) \right \|_{L^2}^2\right)^{\frac{p}{2}}\\
& \leq CA_0^{w(1-\frac{p}{2})}\left(\mu(B(x_0,r))\right)^{1-\frac{p}{2}} \left( \mu(B(x_0,r\right )^{(\frac{1}{2}-\frac{1}{p})p} \\
& \leq C.
\end{split}
\]
To estimate $II$, we show that there a constant $C$ such that for  $x \in (B(x_0,2A_0r))^c$,
\begin{equation}\label{pointwise}S(a)(x) \leq C\mu(B(x_0,r))^{1-\frac{1}{p}}\Big(\frac{r}{d(x,x_0)}\Big)^\eta{1\over  V(x,x_0)}.
\end{equation}
Assuming the estimate in \eqref{pointwise} for the moment, then
\[
\begin{split}
II&\leq C \mu(B(x_0,r))^{p-1}r^{p\eta }\int_{B(x_0,2A_0r)^c}d(x,x_0)^{-p\eta}\Big(\frac{1}{V(x,x_0)}\Big)^p d\mu(x)
\\
&\leq  C \mu(B(x_0,r))^{p-1}r^{p\eta}\sum\limits_{k=1}^\infty  \int_{2^kA_0r<d(x,x_0)\leq 2^{k+1}A_0r}2^{-kp\eta}r^{-p\eta}\Big(\frac{1}{\mu(B(x_0,2^kA_0r))}\Big)^pd\mu(x)
\\
&\leq C \mu(B(x_0,r))^{p-1}\sum\limits_{k=1}^\infty 2^{-kp\eta}\big(\mu(B(x_0,2^kA_0r))\big)^{1-p}
\\
&\leq C \mu(B(x_0,r))^{p-1}\sum\limits_{k=1}^\infty 2^{-k(p\eta+\omega(p-1))}(\mu(B(x_0,r)))^{1-p}
\\
&\leq C ,
\end{split}
\]
where the doubling property on the measure $\mu$ and the fact that $\frac{\omega}{\omega+\eta}<p\leq 1$ are used for the last two inequalities, respectively. The proof of claim \eqref{def:claim} is concluded. Therefore, we only need to show the estimate in \eqref{pointwise}. Note that, by Lemma 3.6 in \cite{HLW}, for fixed $x$ and $k, D_k(x,y),$ as the function of the variable of $y,$ belongs to $\GGs(\eta,\gamma)$ for any $\gamma>0.$ If  $x \in (B(x_0,2A_0r))^c$, using the cancellation condition on $a$ and smoothness condition on the kernel of $D_k$ with the second variable,
\begin{eqnarray*}
|D_k a(x)|
&=& \Big|\int(D_k(x,y)-D_k(x,x_0))a(y)d\mu(y)\Big|
\\
&\leq&  C \int_{B(x_0,r)}\Big({d(y,x_0) \over \delta^k + d(x,x_0)}\Big)^\eta
                {1\over V_{\delta^k}(x) + V(x,x_0)}
                \Big({\delta^k\over \delta^{k} + d(x,x_0)}\Big)^{\gamma}|a(y)|d\mu(y)
\\
&\leq&  C \Big({r \over \delta^k + d(x,x_0)}\Big)^\eta
                {1\over V_{\delta^k}(x) + V(x,x_0)}
                \Big({\delta^k\over \delta^{k} + d(x,x_0)}\Big)^{\gamma}\mu(B(x_0,r))^{1-\frac{1}{p}},
\end{eqnarray*}
where the facts that $d(y,x_0)\leq r\leq \frac{1}{2A_0}d(x,x_0)$ and the size condition on $a$ are used in the first and the last inequalities, respectively. This implies that if $x \in (B(x_0,2A_0r))^c$,
\begin{eqnarray*}
S_c(a)(x)
&=& \Big\{ \sum_{k}|D_k(a)(x)|^2 \Big\}^{1/2}
\\
&\leq&  C \mu(B(x_0,r))^{1-\frac{1}{p}}\Big\{ \sum_{k}\Big|\Big({r \over \delta^k + d(x,x_0)}\Big)^\eta
                {1\over V_{\delta^k}(x) + V(x,x_0)}
                \Big({\delta^k\over \delta^{k} + d(x,x_0)}\Big)^{\gamma}\Big|^2 \Big\}^{1/2}
                \\
&\leq&  C \mu(B(x_0,r))^{1-\frac{1}{p}}\Big(\frac{r}{d(x,x_0)}\Big)^\eta{1\over  V(x,x_0)}\Big\{ \sum_{\delta^k\leq  {d(x,x_0)}}
                                \Big({\delta^k\over  d(x,x_0)}\Big)^{2\gamma} \Big\}^{1/2}\\
&&+  C \mu(B(x_0,r))^{1-\frac{1}{p}}\frac{r^\eta}{V(x,x_0)}\Big\{ \sum_{\delta^k> {d(x,x_0)}}
                 \delta^{-2\eta k}\Big\}^{1/2}
\\
&\leq&   C\mu(B(x_0,r))^{1-\frac{1}{p}}\Big(\frac{r}{d(x,x_0)}\Big)^\eta{1\over  V(x,x_0)}.
\end{eqnarray*}
The proof of the estimate in \eqref{pointwise} is concluded and hence, the proof for $\|f\|_{H^p}\lesssim \|f\|_{H^p_{cw}}$ is complete.

\subsection{A new Calder\'{o}n-type reproducing formula}\label{subsec: A new Calder$o$n-type reproducing formula}

To show $\|f\|_{H^p_{cw}}\lesssim \|f\|_{H^p},$ we observe that the wavelet reproducing formula is not available since the wavelets $\psi^k_\alpha(x)$ have no compact supports. To overcome this problem, a crucial idea is to establish a new kind of Calder\'on-type reproducing formula. For this purpose, we need a result in \cite{MS1}. To be precise, in \cite{MS1}, Mac\'ias and Segovia proved that for any space of
homogeneous type $(X,d,\mu)$ in the sense of Coifman and Weiss, there exists
a quasi-metric $d'$ with the H\"older regularity, which
is geometrically equivalent to the original quasi-metric $d.$ To be more precise,
we state their result as follows.
\begin{theorem}[\cite{MS1}]
Let $d(x,y)$ be a quasi-metric on a set $X$. Then there exists a
quasi-metric $d'(x,y)$ on $X$, a finite constant $C$ and a number
$\theta\in(0,1)$ such that

{\rm(i)} $d'(x,y)$ is geometrically equivalent to $d(x,y)$, that is, $d'(x,y)\approx d(x,y)$ for all $x,y\in X,$  and

{\rm(ii)} for every $x,y$ in $X$ and $r>0$,
$$ |d'(x,y)-d'(x,z)|\leq Cr^{1-\theta}d'(x,y)^\theta $$
\hskip1cm holds whenever $d'(x,y)$ and $d'(x,z)$ are both smaller
than~$r$.
\end{theorem}

We now establish a new Calder\'on-type reproducing formula on space of homogeneous type $(X,d',\mu).$ To do this, we will apply Coifman's construction for an approximation to the identity on $(X, d',\mu).$ More precisely,
let $h\in C^1(\mathbb{R})$ be such that $h(t)=1$ if $|t|\leq 1, h(t)=0$ if $|t|\geq \delta^{-1}$, and $0\leq h(t)\leq 1$ for all $t\in\mathbb{R}$. For any $k\in\mathbb{Z}$, we define
$$
   T_k(f)(x)= \int_X h( \delta^{-k}d'(x,y) )f(y)d\mu(y).
$$
Obviously, we have $ V_{\delta^{k}}(x)\leq T_k(1)(x)\leq V_{\delta^{-1+k}}(x),$ that is, $T_k(1)(x)\sim V_{\delta^{k}}(x).$ By the quasi metric $d'$ and the doubling property on $\mu,$ it is easy to see that for any fixed constant $c$ and $r>0,$ if $d'(x,y)\leq cr$ then $V_r(x)\sim V_r(y).$ Therefore,
$$
  T_k\Big( {1\over T_k(1)} \Big)(x)=\int_X h(\delta^{-k}d'(x,y)) {1\over T_k(1)(y)}d\mu(y)\sim 1.
$$
We now define two multiplication operators by $M_k(f)(x)=\big( T_k(1)(x)\big)^{-1}f(x)$ and $W_k(f)(x)=\big(T_k( {1\over T_k(1)})(x)\big)^{-1}f(x)$ and operators $S_k(f)(x)=M_kT_kW_kT_kM_k(f)(x).$ We claim that $S_k(x,y),$ the kernel of $S_k,$ satisfies the following conditions:
\begin{eqnarray*}
    &\textup{(i)}& S_k(x,y) = 0 {\rm\ for\ } d'(x,y) \geq C\delta^{k},
        {\rm\ and\ } \| S_k\|_{\infty} \leq C {1\over V_{\delta^{k}}(x)+ V_{\delta^{k}}(y)},\\
    &\textup{(ii)}& |S_k(x,y)-S_k(x',y)|
        \leq C\Big({d'(x,x')\over \delta^k}\Big)^\theta {1\over V_{\delta^{k}}(x)+ V_{\delta^{k}}(y)},\\
    &\textup{(iii)}& |S_k(x,y)-S_k(x,y')|
        \leq C\Big({d'(y,y')\over \delta^k}\Big)^\theta {1\over V_{\delta^{k}}(x)+ V_{\delta^{k}}(y)},\\
    &\textup{(iv)}& |[S_k(x,y)-S_k(x,y')]-[S_k(x',y)-S_k(x',y')]|\\
    \ \ \ \ \ \ \ \ \ \ \ \ \ \  \ &
        &\leq C\Big({d'(x,x')\over \delta^k}\Big)^\theta\Big({d'(y,y')\over \delta^k}\Big)^\theta {1\over V_{\delta^{k}}(x)+ V_{\delta^{k}}(y)},\\
    &\textup{(v)}& \int_{X}S_k(x,y) d\mu(y)
        = \int_{X}S_k(x,y) d\mu(x)=1.
    \end{eqnarray*}

To verify the above claim, note that $S_k(x,y)=S_k(y,x)$ and $S_k(1)=1.$ Thus, (v) holds and we only need to check (i), (ii) and (iv). We first check (i) and write
\begin{eqnarray*}
S_k(x,y)={1\over T_k(1)(x)}\bigg\{ \int_X h(\delta^{-k}d'(x,z)) {1\over T_k\Big( {1\over T_k(1)} \Big)(z) } h(\delta^{-k}d'(z,y))d\mu(z) \bigg\}{1\over T_k(1)(y)}.
\end{eqnarray*}
By the condition on the support for the function $h, S_k(x,y)\not=0,$ then $d'(x,y)\leq 2A_0\delta^{-1+k}.$ This implies (i) with the constant $C=2A_0(\delta^{-1}+1).$ To see (ii), by (i) we only need to consider $d'(x,x')\leq \delta^k.$ Indeed, if $d'(x,x')>\delta^{k}$ then (ii) follows directly from (i). To show (ii) for the case that $d'(x,x')\leq
\delta^k,$ we write
\begin{eqnarray*}
&&S_k(x,y)-S_k(x',y)\\
&=&\Big[{1\over T_k(1)(x)}-{1\over T_k(1)(x')}\Big]\bigg\{ \int_X h(\delta^{-k}d'(x,z)) {1\over  T_k\Big( {1\over T_k(1)} \Big)(z) } h(\delta^{-k}d'(z,y)) d\mu(z) \bigg\}{1\over T_k(1)(y)}\\
&&\hskip.2cm+{1\over T_k(1)(x')}\\
&&\hskip.6cm\times \bigg\{ \int_X [h(\delta^{-k}d'(x,z))-h(\delta^{-k}d'(x',z))] {1\over  T_k\Big( {1\over T_k(1)} \Big)(z) } h(\delta^{-k}d'(z,y)) d\mu(z) \bigg\}{1\over T_k(1)(y)}\\
&=:& Z_1+Z_2.
\end{eqnarray*}
For $d'(x,x')\leq\delta^k,$ by the regularity on the function $h,$ we have
\begin{eqnarray*}
  \Big|{1\over T_k(1)(x)}-{1\over T_k(1)(x')}\Big|\lesssim {1\over V_{\delta^{k}}(x)^2} (\delta^{-k} d'(x,x'))^\theta \mu(B(x,2A_0\delta^{-1+k})),
  \end{eqnarray*}
which implies that
$|Z_1|\lesssim (\frac{d(x,x')}{\delta^k})^\theta{1\over V_{2^{-k}}(x) + V_{2^{-k}}(y)}.$

For $Z_2$, observe that if $d'(x',x)\leq\delta^k$ and $ h( \delta^{-k}d'(x,z)) - h( \delta^{-k} d(x',z) ) \not=0,$ then $d(x,z)< A_0(\delta^{-1}+1)\delta^k$. Thus, by the mean value theorem, we obtain
\begin{eqnarray*}
&&\bigg| \int_X [h(\delta^{-k}d'(x,z)) - h( \delta^{-k} d'(x',z) )] {1\over  T_k\Big( {1\over T_k(1)} \Big)(z) } h(\delta^{-k}d'(z,y)) d\mu(z) \bigg|\\
&&\lesssim (\delta^{-k}d'(x,x'))^\theta\mu(B(x,A_0(\delta^{-1}+1)\delta^k),
\end{eqnarray*}
which together with the doubling property on $\mu$ implies that
$$|Z_2|\lesssim { (\delta^{-k}d(x,x'))^\theta \over V_{\delta^{k}}(x)} \sim \Big(\frac{d'(x,x')}{ \delta^k}\Big)^\theta{1\over V_{2^{-k}}(x) + V_{2^{-k}}(y)}.$$
These estimates for $Z_1$ and $Z_2$ implies (ii).

Finally, note that
\begin{eqnarray*}
&&|[S_k(x,y)-S_k(x',y)]-[S_k(x,y')-S_k(x',y')]|\\
&&=\Big[{1\over T_k(1)(x)}-{1\over T_k(1)(x')}\Big]\\&&\ \ \ \ \times \bigg\{ \int_X h(2^kd(x,z)) {1\over  T_k\Big( {1\over T_k(1)} \Big)(z) } h(2^kd(z,y)) d\mu(z) \bigg\}\Big[{1\over T_k(1)(y)}-{1\over T_k(1)(y')}\Big].
\end{eqnarray*}
Repeating the similar proof for (ii) gives (iv).

We observe that $S_k(x,y)$ have compact support with respect to the quasi metric $d'.$ However, the quasi metric $d$ is geometrically equivalent to $d'$ and hence $S_k(x,y)$ have also compact supports with respect to $d.$ This observation will be used for showing the support condition of atoms in the proof for $\|f\|_{H^p_{cw}}\lesssim \|f\|_{H^p}.$

Now we are ready to establish a new Cader\'on-type reproducing formula.
\begin{prop}\label{ new Caderon-type reproducing formula}
     Let $D_k:=S_{k+1}-S_k.$ Then for each given $f\in L^2(X)\cap H^p(X),\frac{\omega}{\omega +\eta}<p\leq 1,$ there exists a unique function $g\in L^2(X)\cap H^p(X)$ such that
    $\|f\|_{L^2}\sim \|g\|_{L^2}, \|f\|_{H^p}\sim \|g\|_{H^p}$ and
\begin{eqnarray}\label{eqnew Caderon-type reproducing formula}
f(x)=\sum_{k}\sum_{\alpha\in \mathscr{X}^{k+N}}
\mu(Q_\alpha^{k+N})D_k(x,x_\alpha^{k+N}){\widetilde D}_k(g)(x_\alpha^{k+N}).
\end{eqnarray}
where the series converges in $L^2(X)\cap H^p(X), N$ is a large fixed integer and ${\widetilde D}_k
=\sum_{|j|\leq N}D_{k+j}.$
\end{prop}
Note that in the wavelet expression given in Theorem A, for each $k\in\Z$ the sum
runs over the set $\alpha\in \mathscr{Y}^k$, while in this new Calder\'on-type reproducing formula, for
each $k\in\Z$ the sum runs over the set $\alpha\in\mathscr{X}^{k +
N}$. Besides the distinction between $\mathscr{Y}$ and
$\mathscr{X}$, the main difference here is that in the
the wavelet expressions the function $f$ is involved on both sides but rather, in this new Calder\'on-type reproducing formula, functions $f$ and $g$ are involved on both sides, respectively. However, $D_k$ involved in this new reproducing formula has compact support which will be important and used frequently. Finally, in this new reproducing
formula, we must sum over all cubes at the smaller scale $k + N$.

We now show Proposition \ref{ new Caderon-type reproducing formula}. Note that the family of operators $S_k$ constructed above is an approximation to the identity in $L^2(X).$ Thus, applying Coifman's decomposition for a fixed positive integer $N$ and $f\in L^2(X),$
\begin{align*}
    f(x)
    &=\sum_{l}D_l(f)(x)=\sum_{l}\sum_{k}D_lD_k(f)(x)\\
    &=\sum_k\sum_{l:\ |k-l|\leq N}D_lD_k(f)(x)+
        \sum_k\sum_{l:\ |k-l|> N}D_lD_k(f)(x)\\
    &=\sum_k\sum_{\alpha\in {\mathscr X}^{k+N}}\mu({ Q}^{k+N}_\alpha)
        D_k(x, x^{k+N}_\alpha){\widetilde D}_k(f)(x_\alpha^{k+N})\\
    &\hskip.5cm + \Big(\sum_{k}D_k{\widetilde D}_k(f)(x)
        - \sum_k\sum_{\alpha\in {\mathscr X}^{k+N}}
        \mu({ Q}^{k+N}_\alpha)D_k(x, x^{k+N}_\alpha)
        {\widetilde D}_k(f)(x_\alpha^{k+N})\Big)
     \\&\hskip.5cm+ \sum_k\sum_{l:\ |k-l|> N}D_kD_l(f)(x)\\
    &=: T_{N}(f)(x) + R^{(1)}_{N}(f)(x) + R^{(2)}_{N}(f)(x)
\end{align*}
in the sense of $L^2(X)$, where ${\widetilde D}_k=\sum_{|j|\leq N}D_{k+j}.$ and particularly, as mentioned, the kernel of $D_k$ has compact support.

Note that $T_N = I - R^{(1)}_N-R^{(2)}_N$ by definition and $D_k(x,y),$ the kernels of $D_k,$ satisfy the decay condition (i), particularly, as mentioned, the kernels of $D_k$ have compact supports, the smoothness conditions (ii)--(iv) and the moment conditions $\int_{X}D_k(x,y) d\mu(y)= \int_{X}D_k(x,y) d\mu(x)=0.$ Therefore the Cotlar--Stein lemma can be applied to show that $R^{(i)}_N$ as well as $T_N$ are bounded on $L^2(X).$ Moreover, we will show that for $f\in L^2(X), \frac{\omega}{\omega+\eta}<p<\infty$ and $i=1,2,$
\begin{equation}\label{eqn:S(R_N)_Lp_estimate}
    \|S(R^{(i)}_N(f))\|_{L^p}
    \leq C\delta^{\theta N}\|S(f)\|_{L^p}.
\end{equation}
This will imply that if $N$ is chosen so that $2C\delta^{\theta N}<1,$ then $(T_N)^{-1},$ the inverse of $T_N,$ is bounded on $L^2(X)\cap H^p(X).$ Thus, given $f\in L^2(X)\cap H^p(X),$ set $g=(T_N)^{-1}f.$ Then $g$ satisfies all conditions in Proposition \ref{ new Caderon-type reproducing formula} and moreover, the representation of $f$ in \eqref{eqnew Caderon-type reproducing formula} holds. We only prove the estimate in \eqref{ new Caderon-type reproducing formula} for $R^{(2)}_N$ since the proof for $R^{(1)}_N$ is similar to one given in \cite{HLW}. See pages 34--39 in \cite{HLW} for the details. To estimate $\|S(R^{(2)}_N(f))\|_{L^p(X)}$, we write
\begin{equation*}
    \|S(R^{(2)}_N(f))\|_{L^p}
    = \Big\|\Big\{ \sum_{k}\sum_{\alpha\in\mathscr{Y}^k}
        \big| \langle \psi_{\alpha}^{k},R^{(2)}_N(f) \rangle
        \widetilde{\chi}_{{Q}_{\alpha}^{k}}(\cdot)
        \big|^2 \Big\}^{1/2} \Big\|_{L^p}.
\end{equation*}
By the $L^2(X)$-boundedness of $R^{(2)}_N$ and the wavelet
reproducing formula in Theorem A for
$f\in L^2(X)$, we have
\begin{eqnarray*}
    &&\Big\langle {\psi_{\alpha}^{k}},
        R^{(2)}_N(f) \Big\rangle=\sqrt{\mu({Q}_{\alpha}^{k})}\Big\langle {\psi_{\alpha}^{k}\over \sqrt{\mu({Q}_{\alpha}^{k})}}, R^{(2)}_N(f) \Big\rangle \\
    &&= \sum_{k'}
    \sum_{\alpha'\in {\mathscr Y}^{k'}}
    \sqrt{\mu({Q}_{\alpha}^{k})}
    \sqrt{\mu({Q}_{\alpha^{'}}^{k^{'}})}
    \Big\langle {\psi_{\alpha}^{k}(\cdot)\over \sqrt{\mu({Q}_{\alpha}^{k})}},\big\langle \sum_{|j-l|>N} D_jD_l(\cdot,\cdot), {\psi_{\alpha^{'}}^{k^{'}}(\cdot) \over \sqrt{\mu({Q}_{\alpha^{'}}^{k^{'}})} }\big\rangle\Big\rangle \big\langle \psi_{\alpha^{'}}^{k^{'}}, f\big\rangle.\nonumber
\end{eqnarray*}
To simplify the notation, set $E_k^\alpha(x_\alpha^k,x)={\psi_{\alpha}^{k}(x)\over \sqrt{\mu({Q}_{\alpha}^{k})}}.$ We now estimate the term $$\big\langle E_k^\alpha(\cdot),\langle \sum_{|j-l|>N} D_jD_l(\cdot,\cdot), E_{k'}^{\alpha'}(\cdot)\rangle\big\rangle.$$ Observe first that
$E_k^\alpha(x_\alpha^k,x)$ is a test function in $\GGs(x_\alpha^k, \delta^k,\eta, \gamma)$ for any $\gamma>0$ and $D_j(x,y)$ is a test function in $\GGs(x,\delta^j,\theta,\gamma)$ if $x$ is fixed or in $\GGs(y,\delta^j,\theta,\gamma)$ if $y$ is fixed. Furthermore, by standard almost orthogonal estimate, $D_jD_l(x,y)$ satisfies the estimates (i)--(iii) as $D_{j\wedge l}(x,y)$ does with the bounds $C\delta^{|j-l|\theta},$ where, as usual, $j\wedge l=min \lbrace j,l\rbrace$ denotes the minimum of $j$ and $l.$ Indeed, if $j\leq k,$ observing that $D_jD_l(x,y)=\int [D_j(x,z)-D_j(x,y)]D_l(z,y)d\mu(z)$ and applying the smoothness condition on $D_j$ and the size condition on $D_l$ yields that $D_jD_l(x,y)$ satisfies the condition (i) with the constant replaced by $C\delta^{|j-l|\theta}.$ Writing $D_jD_l(x,y)-D_jD_l(x',y)=\int \lbrace [D_j(x,z)-D_j(x,y)]-[D_j(x',z)-D_j(x',y)]\rbrace D_l(z,y)d\mu(z)=\int \lbrace [D_j(x,z)-D_j(x',z)]-[D_j(x,y)-D_j(x',y)]\rbrace D_l(z,y)d\mu(z)$ and applying the smoothness condition (iv) on $D_j$ and the size condition on $D_l$ implies that $D_jD_l(x,y)$ satisfies the condition (ii) with the bound replaced by $C\delta^{|j-l|\theta}.$ We point out that the condition (iv) was not used in \cite{DJS} for the proof of the $L^2$ boundedness but this is crucial for the boundedness of $H^p(X).$ We leave the details of the proof to the reader.
Set
$$F_{k'}(x, x_{k'}^{\alpha'})=\sum_{(j,l): |j-l|>N}\langle D_jD_l(x,\cdot), E_{k'}^{\alpha'}(\cdot)\rangle
=\sum_{(j,l): |j-l|>N}\int D_jD_l(x,y)E_{k'}^{\alpha'}(y,x_{k'}^{\alpha'})d\mu(y).$$
We claim that
\begin{eqnarray*}
    &&\textup{(a)}\ \ |F_{k^{'}}(x, x_{\alpha^{'}}^{k^{'}})|
        \leq C \delta^{\theta N}
            {1\over V_{\delta^{k^{'}}}(x) + V(x,x_{\alpha^{'}}^{k^{'}})}
            \Big({\delta^{k^{'}}\over \delta^{k^{'}}
            + d(x,x_{\alpha^{'}}^{k^{'}})}\Big)^{\gamma}, \hspace{4.2cm}
\end{eqnarray*}
for all $\gamma \in (0,\eta)$, and
\begin{eqnarray*}
    \lefteqn{\textup{(b)}\ \
            |F_{k^{'}}(x, x_{\alpha^{'}}^{k^{'}}) - F_{k^{'}}(x',x_{\alpha^{'}}^{k^{'}})|} \hspace{3cm}\\
    &&\leq C\delta^{\theta N} \Big({d(x,x')\over \delta^{k'}}\Big)^{\eta'}
    {}\times
        \bigg[{1\over V_{\delta^{k^{'}}}(x) + V(x,x_{\alpha^{'}}^{k^{'}})}
        \Big({\delta^{k^{'}}\over \delta^{k^{'}} + d(x,x_{\alpha^{'}}^{k^{'}})}\Big)^{\gamma} \\
    && \hspace{1cm} {}+ {1\over V_{\delta^{k^{'}}}(x') + V(x',x_{\alpha^{'}}^{k^{'}})}
        \Big({\delta^{k^{'}}\over \delta^{k^{'}}
        + d(x',x_{\alpha^{'}}^{k^{'}})}\Big)^{\gamma}\bigg],
\end{eqnarray*}
for all $\gamma, \eta' \in (0,\eta)$.

The main tool to show the above claim is the following almost-orthogonality estimate: There
exists a constant~$C$ such that
\begin{eqnarray}\label{eqn:almost_orthogonality_estimate}
    &&\Big|\Big\langle
        \frac{\psi_{\alpha}^{k}(\cdot)}{\sqrt{\mu({ Q}_{\alpha}^{k})}},
        D_{j}(\cdot,x)
        \Big\rangle\Big| \nonumber \\
    &&\hskip.5cm \leq C \delta^{\vert k - j\vert{\eta }}
        \frac{1}{V_{\delta^{(j\wedge k)}}(x_\alpha^k)
            + V_{\delta^{(j\wedge k)}}(x)
            + V(x_\alpha^k,x)}
        \Big({{\delta^{(j\wedge k)}}
            \over {\delta^{(j\wedge k)}
            + d(x_\alpha^k,x)}}
        \Big)^{\gamma }.
\end{eqnarray}
See the proof for such an argument in (4.4) on page 31 \cite{HLW}. Applying the almost-orthogonality estimate in \eqref{eqn:almost_orthogonality_estimate} with $D_j$ replaced by $D_jD_l$ implies that
\begin{eqnarray}
    &&\Big|\big\langle
        D_{j}D_l(x,\cdot), E_{k'}^{\alpha'}(\cdot)
        \big\rangle\Big| \nonumber \\
    &&\hskip.5cm \leq C \delta^{|j-l|\theta}\delta^{\vert k' - j\wedge l\vert{\eta }}
        \frac{1}{V_{\delta^{(j\wedge l\wedge k')}}(x_{\alpha'}^{k'})
            + V_{\delta^{(j\wedge l\wedge k')}}(x)
            + V(x_{\alpha'}^{k'},x)}
        \Big({{\delta^{(j\wedge l\wedge k')}}
            \over {\delta^{(j\wedge l\wedge k')}
            + d(x_{\alpha'}^{k'},x)}}
        \Big)^{\gamma }.
\end{eqnarray}
Applying the above estimates for $\eta>\gamma$ gives
\begin{eqnarray*}
&&|F_{k'}(x, x_{k'}^{\alpha'})|=|\sum_{(j,l): |j-l|>N}\langle D_jD_l(x,\cdot), E_{k'}^{\alpha'}(\cdot)\rangle|\\
&&\leq C\sum_{(j,l): |j-l|>N}\delta^{|j-l|\theta}\delta^{\vert k' - j\wedge l\vert{\eta }}
        \frac{1}{V_{\delta^{(j\wedge l\wedge k')}}(x_{\alpha'}^{k'})
            + V_{\delta^{(j\wedge l\wedge k')}}(x)
            + V(x_{\alpha'}^{k'},x)}
        \Big({{\delta^{(j\wedge l\wedge k')}}
            \over {\delta^{(j\wedge l\wedge k')}
            + d(x_{\alpha'}^{k'},x)}}
        \Big)^{\gamma }\\
&&\leq C\sum_{(j,l): |j-l|>N}\delta^{|j-l|\theta}\delta^{\vert k' - j\wedge l\vert{(\eta-\gamma) }}
        \frac{1}{V_{\delta^{k'}}(x)
            + V(x_{\alpha'}^{k'},x)}
        \Big({{\delta^{k'}}
            \over {\delta^{k'}
            + d(x_{\alpha'}^{k'},x)}}
        \Big)^{\gamma }\\
&&\leq C\delta^{N\theta}\frac{1}{V_{\delta^{k'}}(x)
            + V(x_{\alpha'}^{k'},x)}
        \Big({{\delta^{k'}}
            \over {\delta^{k'}
            + d(x_{\alpha'}^{k'},x)}}
        \Big)^{\gamma },
\end{eqnarray*}
which implies $(a).$ Similarly, for $\eta'<\eta$ and $\gamma<\eta,$
\begin{eqnarray*}
&&|F_{k^{'}}(x, x_{\alpha^{'}}^{k^{'}}) - F_{k^{'}}(x',x_{\alpha^{'}}^{k^{'}})|\\
&&=\Big|\sum_{(j,l): |j-l|>N}\langle [D_jD_l(x,\cdot)-D_jD_l(x',\cdot)], E_{k'}^{\alpha'}(\cdot)\rangle\Big|\\
&&\leq C\sum_{(j,l): |j-l|>N}\delta^{|j-l|\theta}\big(\frac{d'(x,x')}{\delta^{j\wedge l}}\big)^\eta
\delta^{\vert k' - j\wedge l\vert{\eta }}
\\&&\hskip.5cm\times\Big[
        \frac{1}{V_{\delta^{(j\wedge l\wedge k')}}(x_{\alpha'}^{k'})
            + V_{\delta^{(j\wedge l\wedge k')}}(x)
            + V(x_{\alpha'}^{k'},x)}
        \Big({{\delta^{(j\wedge l\wedge k')}}
            \over {\delta^{(j\wedge l\wedge k')}
            + d(x_{\alpha'}^{k'},x)}}
        \Big)^{\gamma }\\
       &&\hskip1cm +\frac{1}{V_{\delta^{(j\wedge l\wedge k')}}(x_{\alpha'}^{k'})
            + V_{\delta^{(j\wedge l\wedge k')}}(x')
            + V(x_{\alpha'}^{k'},x')}
        \Big({{\delta^{(j\wedge l\wedge k')}}
            \over {\delta^{(j\wedge l\wedge k')}
            + d(x_{\alpha'}^{k'},x')}}
        \Big)^{\gamma }\Big]\\
&&\leq C\sum_{(j,l): |j-l|>N}\delta^{|j-l|\theta}\big(\frac{d(x,x')}{\delta^{k'}}\big)^{\eta^{'}}
\delta^{\vert k' - j\wedge l\vert{(\eta-\eta') }}\times\Big[
        \frac{1}{ V_{\delta^{k'}}(x)
            + V(x_{\alpha'}^{k'},x)}
        \Big({{\delta^{k'}}
            \over {\delta^{k'}
            + d(x_{\alpha'}^{k'},x)}}
        \Big)^{\gamma }\\
        &&\hskip1cm+\frac{1}{V_{\delta^{k'}}(x')
            + V(x_{\alpha'}^{k'},x')}
        \Big({{\delta^{k'}}
            \over {\delta^{k'}
            + d(x_{\alpha'}^{k'},x')}}
        \Big)^{\gamma }\Big]\\
&&\leq C\delta^{N\theta}\big(\frac{d(x,x')}{\delta^{k^{'}}}\big)^{\eta^{'}}\\
&&\hskip1cm\times\Big[
        \frac{1}{ V_{\delta^{k'}}(x)
            + V(x_{\alpha'}^{k'},x)}
        \Big({{\delta^{k'}}
            \over {\delta^{k'}
            + d(x_{\alpha'}^{k'},x)}}
        \Big)^{\gamma }+\frac{1}{V_{\delta^{k'}}(x')
            + V(x_{\alpha'}^{k'},x')}
        \Big({{\delta^{k'}}
            \over {\delta^{k'}
            + d(x_{\alpha'}^{k'},x')}}
        \Big)^{\gamma }\Big],
\end{eqnarray*}
where the equivalence between $d'$ and $d$ is used in the second inequality. This gives the desired estimate for $(b).$

Since $F_{k'}(\cdot, x_{k'}^{\alpha'})$ satisfies $(a)$ and $(b)$, from Remakr 4.5 in \cite{HLW}, we obtain that
\begin{eqnarray*}
  &&  \Big|\Big\langle E_k^\alpha(x_\alpha^k,\cdot),F_{k'}(\cdot, x_{k'}^{\alpha'})\Big\rangle \Big| \\
  &&  \leq C \delta^{N\theta}\delta^{|k' - k|\eta^{''}}
        \frac{1}{V_{\delta^{(k\wedge k')}}(x_{\alpha'}^{k'})
            + V_{\delta^{(k\wedge k')}}(x_\alpha^k)
            + V(x_{\alpha'}^{k'},x_\alpha^k)}
        \Big({{\delta^{(k\wedge k')}}
            \over {\delta^{(k\wedge k')}
            + d(x_{\alpha'}^{k'},x_\alpha^k)}}
        \Big)^{\gamma }
\end{eqnarray*}
for $\eta^{''}<\eta'$ and $\gamma<\eta'$.

Thus, we have
\begin{eqnarray*}
    \Big\langle {\psi_{\alpha}^{k}},
        R^{(2)}_N(f) \Big\rangle
    &=& \sum_{k'}
    \sum_{\alpha'\in {\mathscr Y}^{k'}}
    \sqrt{\mu({Q}_{\alpha}^{k})}
    \sqrt{\mu({Q}_{\alpha^{'}}^{k^{'}})}
    \Big|\Big\langle E_k^\alpha(x_\alpha^k,\cdot),F_{k'}(\cdot, x_{k'}^{\alpha'})\Big\rangle\Big|\Big| \big\langle \psi_{\alpha^{'}}^{k^{'}}, f\big\rangle\Big|\\
&\leq& C\delta^{N\theta}\sqrt{\mu({Q}_{\alpha}^{k})}\sum_{k'}
    \sum_{\alpha'\in {\mathscr Y}^{k'}}
        \frac{1}{V_{\delta^{(k\wedge k')}}(x_{\alpha'}^{k'})
            + V_{\delta^{(k\wedge k')}}(x_\alpha^k)
            + V(x_{\alpha'}^{k'},x_\alpha^k)}\\
&&\times      \mu({Q}_{\alpha^{'}}^{k^{'}}) \delta^{|k' - k|\eta^{''}}  \Big({{\delta^{(k\wedge k')}}
            \over {\delta^{(k\wedge k')}
            + d(x_{\alpha'}^{k'},x_\alpha^k)}}
        \Big)^{\gamma }\Big| \Big\langle {\psi_{\alpha^{'}}^{k^{'}}\over \sqrt{\mu({Q}_{\alpha^{'}}^{k^{'}})}}, f\Big\rangle\Big|
\end{eqnarray*}
As a consequence, we have
\begin{eqnarray*}
&&\Big\{   \sum_k\sum_{\alpha\in\mathscr{Y}^k}  \Big|\big\langle {\psi_{\alpha}^{k}},
        R^{(2)}_N(f) \big\rangle \widetilde{\chi}_{Q_\alpha^k}(x) \Big|^2\Big\}^{1/2}\\
&&\leq C\delta^{N\theta}\bigg\{   \sum_k\sum_{\alpha\in\mathscr{Y}^k}  \bigg| \sum_{k'}
    \sum_{\alpha'\in {\mathscr Y}^{k'}}    \mu({Q}_{\alpha^{'}}^{k^{'}}) \delta^{|k' - k|\eta^{''}}
        \frac{1}{V_{\delta^{(k\wedge k')}}(x_{\alpha'}^{k'})
            + V_{\delta^{(k\wedge k')}}(x_\alpha^k)
            + V(x_{\alpha'}^{k'},x_\alpha^k)}\\
&&\hskip.5cm\times        \Big({{\delta^{(k\wedge k')}}
            \over {\delta^{(k\wedge k')}
            + d(x_{\alpha'}^{k'},x_\alpha^k)}}
        \Big)^{\gamma }\Big| \Big\langle {\psi_{\alpha^{'}}^{k^{'}}\over \sqrt{\mu({Q}_{\alpha^{'}}^{k^{'}})}}, f\Big\rangle\Big| \bigg|^2  \chi_{Q_\alpha^k}(x)\bigg\}^{1/2}.
\end{eqnarray*}
Then, following the same argument as in page 33 of \cite{HLW}, i.e., the estimate as in \cite{FJ}, pp.147--148  and the Fefferman--Stein vector-valued maximal function inequality in \cite{FS}, we obtain that
\begin{eqnarray*}
    \|S(R^{(2)}_N(f))\|_{L^p}
   & = & \Big\|\Big\{ \sum_{k}\sum_{\alpha\in\mathscr{Y}^k}
        \big| \langle \psi_{\alpha}^{k},R^{(2)}_N(f) \rangle
        \widetilde{\chi}_{{Q}_{\alpha}^{k}}(\cdot)
        \big|^2 \Big\}^{1/2} \Big\|_{L^p}\\
  &\leq & C\delta^{N\theta} \Big\|\Big\{ \sum_{k'}\sum_{\alpha'\in\mathscr{Y}^{k'}}
        \big| \langle \psi_{\alpha^{'}}^{k^{'}}, f \rangle
        \widetilde{\chi}_{{Q}_{\alpha'}^{k'}}(\cdot)
        \big|^2 \Big\}^{1/2} \Big\|_{L^p}\\
  &=& C\delta^{N\theta} \|S(f)\|_{L^p}
\end{eqnarray*}
The proof of Proposition \ref{ new Caderon-type reproducing formula} is concluded.

\subsection{The proof for $\|f\|_{H^p_{cw}}\lesssim \|f\|_{H^p}$}

Now we show that $\|f\|_{H^p_{cw}}\lesssim \|f\|_{H^p}$ for $f\in L^2(X)\cap H^p(X)$ since $L^2(X)\cap H^p(X)$ is dense in $H^p(X).$ To this end, we define a new Littlewood--Paley square function on $L^2(X)\cap H^p(X)$ by
\[
{\widetilde{S}}(f)(x)=\Big(\sum_{k}\sum_{\alpha\in \mathscr{X}^{k +
N}}|{\widetilde D}_k(g)(x_\alpha^{k+N})|^2\chi_{Q_\alpha^{k+N}}(x)\Big)^{1/2},
\]
where $f, g$ are related as given in Proposition \ref{ new Caderon-type reproducing formula}.

By the estimate in (4.2) of Theorem 4.4 in \cite{HLW}, we have that for $\frac{\omega}{\omega +\eta}<p\leq \infty,$ $\|{\widetilde{S}}(f)\|_{L^p}\leq C\|S(g)\|_{L^p}.$ Thus, if $f\in L^2(X)\cap H^p(X)$ and $g$ is given by Proposition \ref{ new Caderon-type reproducing formula}, then $\|{\widetilde{S}}(f)\|_{L^2}\leq C\|g\|_{L^2}\leq C\|f\|_{L^2}$ and $\|{\widetilde{S}}(f)\|_{L^p}\leq C\|g\|_{H^p}\leq C\|f\|_{H^p}, \frac{\omega}{\omega +\eta}<p\leq 1.$ We now set
\[\Omega_l=\left\{ x\in X: {\widetilde{S}}(f)(x)>2^l\right\},\]
\[
B_l=\left\{ Q_\alpha^k:\mu(Q_\alpha^k\cap\Omega_l)>{1\over2}\mu(Q_\alpha^k)
\text{ and }
\mu(Q_\alpha^k \cap\Omega_{l+1})\leq {1\over2}\mu(Q_\alpha^k)\right\}
\]
and
\[\widetilde{\Omega}_l=\left\{ x\in {X}:  M(\chi_{\Omega_l})(x)>1/2\right\},\]
where $M$ is the Hardy--Littlewood maximal function on $X$ and hence, $\mu(\widetilde{\Omega}_l)\leq C\mu({\Omega}_l).$

Note that each dyadic cube $Q_\alpha^k$ belongs to one and only one $B_l$. Applying Proposition \ref{ new Caderon-type reproducing formula}, we can write
\begin{eqnarray*}
f(x)&=&\sum_{l}\sum_{ Q_\alpha^{k+N}\in B_l}  \mu(Q_\alpha^{k+N})D_k(x,x_\alpha^{k+N}) {\widetilde D}_k(g)(x_\alpha^{k+N})\\
&=&\sum_{l} \sum_{j:\widetilde{Q}_\alpha^{j,l} \in B_l} \sum_{\substack{Q_\alpha^{k+N}\in B_l\\ Q_{\alpha}^{k+N}\subseteq
\widetilde{Q}_{\alpha}^{j,l}}} \mu(Q_\alpha^{k+N})D_k(x,x_\alpha^{k+N}) {\widetilde D}_k(g)(x_\alpha^{k+N}),
\end{eqnarray*}
where the series converge in $L^2(X)$ and in $H^p(X),$ and we denote $\widetilde{Q}_\alpha^{j,l} \in B_l$ by the maximal dyadic cubes contained in $B_l$.

Finally, we write
\[
f(x)=\sum_{l} \sum_{j:\widetilde{Q}_\alpha^{j,l} \in B_l} \lambda_l(\widetilde{Q}_{\alpha}^{j,l}) a_{l,\widetilde{Q}_{\alpha}^{j,l}}(x),
\]
where
\[
a_{l,\widetilde{Q}_{\alpha}^{j,l}}(x) ={1\over \lambda_l(\widetilde{Q}_{\alpha}^{j,l}) } \sum_{\substack{Q_\alpha^{k+N}\in B_l\\ Q_{\alpha}^{k+N}\subseteq
\widetilde{Q}_{\alpha}^{j,l}}}  \mu(Q_\alpha^{k+N})D_k(x,x_\alpha^{k+N}){\widetilde D}_k(g)(x_\alpha^{k+N})
\]
and
\[
\lambda_l(\widetilde{Q}_{\alpha}^{j,l})= \widetilde{C}\Big(
 \sum_{\substack{Q_\alpha^{k+N}\in B_l\\ Q_{\alpha}^{k+N}\subseteq
\widetilde{Q}_{\alpha}^{j,l}}}  \mu(Q_\alpha^{k+N}) |{\widetilde D}_k(g)(x_\alpha^{k+N}) |^2 \Big)^{1/2} \mu(\widetilde{Q}_{\alpha}^{j,l})^{1/p-1/2}
\]
with a constant $\widetilde{C}$ to be determined later.

Now we prove that the above decomposition of $f$ is a desired atomic decomposition. To see this, we first show that each $a_{l,\widetilde{Q}_{\alpha}^{j,l}}(x)$ is an $(p,2)$ atom. Note that the quasi metrics $d$ and $d'$ are geometrically equivalent. This together with the fact, as mentioned, that $D_k(x,x_\alpha^k)$ have compact supports implies that $a_{l,\widetilde{Q}_{\alpha}^{j,l}}(x)$ is supported in $C\widetilde{Q}_{\alpha}^{j,l}$ with $C=C_1A_0\delta^{-N}(\delta^{-1}+1),$ where $C_1$ is the constant appeared in the equivalence between $d$ and $d'.$
By the cancellation conditions on $D_k(x,x_\alpha^k),$ it is easy to see that
\[
\int _X a_{l,\widetilde{Q}_{\alpha}^{j,l}}(x) d\mu(x)=0.
\]
To verify the size condition of $a_{l,\widetilde{Q}_{\alpha}^{j,l}}(x)$, applying the duality argument implies that
\begin{eqnarray*}
\lefteqn{\|a_{l,\widetilde{Q}_{\alpha}^{j,l}}(x)\|_{L^2}  }\\
&=&\sup _{h \in L^2(X),\  \|h\|_{L^2}=1 } \Big|\big\langle a_{l,\widetilde{Q}_{\alpha}^{j,l}}(x), h(x)  \big\rangle \Big|\\
&=&\sup _{h \in L^2(X),\  \|h\|_{L^2}=1 } \Big| {1\over \lambda_l(\widetilde{Q}_{\alpha}^{j,l}) } \sum_{\substack{Q_\alpha^{k+N}\in B_l\\ Q_{\alpha}^{k+N}\subseteq
\widetilde{Q}_{\alpha}^{j,l}}}  \mu(Q_\alpha^{k+N})D_k(h)(x_\alpha^{k+N}){\widetilde D}_k(g)(x_\alpha^{k+N})  \Big|\\
&\leq& \sup _{h \in L^2(X),\  \|h\|_{L^2}=1 } {1\over \lambda_l(\widetilde{Q}_{\alpha}^{j,l}) } \Big(\sum_{\substack{Q_\alpha^k\in B_l\\ Q_{\alpha}^{k+N}\subseteq
\widetilde{Q}_{\alpha}^{j,l}}}  \mu(Q_\alpha^{k+N}) |D_k(h)(x_\alpha^{k+N})|^2\Big)^{1/2}\\
&&\hskip1cm\times\Big(
 \sum_{\substack{Q_\alpha^{k+N}\in B_l\\ Q_{\alpha}^{k+N}\subseteq
\widetilde{Q}_{\alpha}^{j,l}}}  \mu(Q_\alpha^{k+N}) |{\widetilde D}_k(g)(x_\alpha^{k+N}) |^2 \Big)^{1/2}\\
&\leq& \sup _{h \in L^2(X),\ |h\|_{L^2}=1 }  \widetilde{C} \|h\|_{L^2} {1\over \lambda_l(\widetilde{Q}_{\alpha}^{j,l}) } \Big(
 \sum_{\substack{Q_\alpha^{k+N}\in B_l\\ Q_{\alpha}^{k+N}\subseteq
\widetilde{Q}_{\alpha}^{j,l}}}  \mu(Q_\alpha^{k+N}) |{\widetilde D}_k(g)(x_\alpha^{k+N}) |^2 \Big)^{1/2}\\
&\leq& \mu(\widetilde{Q}_{\alpha}^{j,l})^{1/2-1/p},
\end{eqnarray*}
where $\widetilde{C}$ is chosen to be the constant satisfying the following estimate:
\begin{eqnarray*}
\Big(\sum_{k}\sum_{\alpha\in \mathscr{X}^{k+N}}  \mu(Q_\alpha^{k+N}) |D_k(h)(x_\alpha^{k+N})|^2\Big)^{1/2}
&\leq&   \widetilde{C} \|h\|_{L^2(X)}
\end{eqnarray*}
for each $h\in L^2(X)$.

We finally have
\begin{eqnarray*}
\lefteqn{\sum_l\sum_{j: \widetilde{Q}_{\alpha}^{j,l} } |\lambda_{l,\widetilde{Q}_{\alpha}^{j,l}}|^p  }\\
&=& \sum_l\sum_{ \widetilde{Q}_{\alpha}^{j,l} \in B_l}  \widetilde{C}^p\Big(
 \sum_{\substack{Q_\alpha^{k+N}\in B_l\\ Q_{\alpha}^{k+N}\subseteq
\widetilde{Q}_{\alpha}^{j,l}}}  \mu(Q_\alpha^{k+N}) |{\widetilde D}_k(g)(x_\alpha^{k+N}) |^2 \Big)^{p/2} \mu(\widetilde{Q}_{\alpha}^{j,l})^{1-p/2}\\
&\leq& \widetilde{C}^p \sum_l   \Big( \sum_{ \widetilde{Q}_{\alpha}^{j,l} \in B_l }
 \sum_{\substack{Q_\alpha^{k+N}\in B_l\\ Q_{\alpha}^{k+N}\subseteq
\widetilde{Q}_{\alpha}^{j,l}}}  \mu(Q_\alpha^{k+N}) |{\widetilde D}_k(g)(x_\alpha^{k+N}) |^2 \Big)^{p/2}  \Big(\sum_{ \widetilde{Q}_{\alpha}^{j,l}\in B_l } \mu(\widetilde{Q}_{\alpha}^{j,l}) \Big)^{1-p/2}\\
&\leq& \widetilde{C}^p \sum_l  \Big(
 \sum_{Q_\alpha^{k+N}\in B_l}  \mu(Q_\alpha^{k+N}) |{\widetilde D}_k(g)(x_\alpha^{k+N}) |^2 \Big)^{p/2} \mu(\widetilde{\Omega}_l)^{1-p/2},
\end{eqnarray*}
where the fact that if $\widetilde{Q}_{\alpha}^{j,l}\in B_l$ then $\widetilde{Q}_{\alpha}^{j,l}\subset \widetilde{\Omega}_l$ is used in the last inequality.
Note that
\begin{eqnarray*}
\int_{\widetilde{\Omega}_l\backslash \Omega_{l+1}} {\widetilde{S}}(f)(x)^2 d\mu(x)
\leq 2^{(l+1)2}\mu(\widetilde{\Omega}_l)\leq C2^{2l}\mu(\Omega_l).
\end{eqnarray*}
While we also have
\begin{eqnarray*}
\int_{\widetilde{\Omega}_l\backslash \Omega_{l+1}} {\widetilde{S}}(f)(x)^2 d\mu(x)
&=&    \int_{\widetilde{\Omega}_l\backslash \Omega_{l+1}} \sum_{k}\sum_{\alpha\in \mathscr{X}^{k+N}}
 |{\widetilde D}_k(g)(x_\alpha^{k+N})|^2\chi_{Q_\alpha^{k+N}}(x)  d\mu(x)\\
&\geq& \sum_{Q_{\alpha}^{k+N}\in B_l}
 |{\widetilde D}_k(g)(x_\alpha^{k+N})|^2 \mu(Q_\alpha^{k+N} \cap (\widetilde{\Omega}_l\backslash \Omega_{l+1}))\\
&\geq& {1\over 2} \sum_{Q_\alpha^{k+N}\in B_l}  \mu(Q_\alpha^{k+N}) |{\widetilde D}_k(g)(x_\alpha^{k+N}) |^2.
\end{eqnarray*}
As a consequence, we obtain
\begin{eqnarray*}
 \sum_{Q_\alpha^k\in B_l}  \mu(Q_\alpha^k) |{\widetilde D}_k(g)(x_\alpha^k) |^2\leq  C2^{2l}\mu(\Omega_l),
\end{eqnarray*}
which implies that
\begin{eqnarray*}
\sum_l\sum_{j: \widetilde{Q}_{\alpha}^{j,l}} |\lambda_{l,\widetilde{Q}_{\alpha}^{j,l}}|^p
&\leq& \widetilde{C}^p \sum_l  2^{lp} \mu(\Omega_l)^{p/2} \mu(\Omega_l)^{1-p/2}\\
&\leq& C\|\widetilde{S}(f)\|_{L^p}^p
\leq C\|f\|_{H^p}^p.
\end{eqnarray*}

We have proved that if $f\in L^2(X)\cap H^p(X),$ then $f$ has an atomic decomposition. But we still need to show that the atomic decomposition for $f$ obtained above must converge in the dual of $\mathcal C_{\frac{1}{p}-1}(X).$ This follows from the following
\begin{lemma}\label{lipcmop6}
    Suppose that $(X,d,\mu)$ is space of homogeneous type in the sense of Coifman and Weiss. Let $\omega$
    is the upper dimension of~$(X,d,\mu)$ and $\frac{\omega}{\omega+\eta}<p\leq 1,$ where $\eta$ is the H\"{o}lder
    exponent of wavelets $\psi_{\alpha}^{k}.$ Then $\mathcal C_{\frac{1}{p}-1}(X)\subset \cmo^p(X).$ Moreover, there exists a constant $C$ such that
    $$\parallel f\parallel_{\cmo^p}\leq C\parallel f\parallel_{\mathcal C_{\frac{1}{p}-1}}.$$
\end{lemma}
    Assuming this lemma for the moment, we show a general argument that if $f$ has an atomic decomposition which converges in the norm of $H^p(X)$ then this decomposition also converges in the dual of ${\mathcal C_{\frac{1}{p}-1}(X)}.$ Indeed, suppose that $f$ has an atomic decomposition with $f=\sum_{j\geq 1}\lambda_j a_j(x)$ where the series converges in the norm of $H^p(X).$ Let $g$ be in ${\mathcal C_{\frac{1}{p}-1}(X)}.$ By Lemma \ref{lipcmop6}, $g\in \cmo^p(X)$ and $\|g\|_{\cmo^p}\leq C\|g\|_{\mathcal C_{\frac{1}{p}-1}}.$ Thus, by the duality in Theorem C,
    $$|<f-\sum_{j\leq N}\lambda_j a_j, g>|\leq C\|f-\sum_{j\leq N}\lambda_j a_j\|_{H^p}\|g\|_{\cmo^p}\leq C\|f-\sum_{|j|\leq N}\lambda_j a_j\|_{H^p}\|g\|_{\mathcal C_{\frac{1}{p}-1}},$$
    where the last term tends to zero as $N$ tends to infinity since $\|f-\sum_{j\leq N}\lambda_j a_j\|_{H^p}$ tends to zero as $N$ tends to infinity. This implies that $\sum_{j\geq 1}\lambda_j a_j(x)$ converges to $f$ in the dual of $\mathcal C_{\frac{1}{p}-1}(X).$ Note that what we have proved is that if $f\in H^p(X)$ then $f$ has an atomic decomposition and moreover, this decomposition converges in both of the norm of $L^2(X)$ and the norm of $H^p(X).$ Therefore, by the above general argument, the atomic decomposition of $f$ also converges in the dual of $\mathcal C_{\frac{1}{p}-1}(X)$ and hence $f$ belongs to the atomic Hardy space $H^p_{cw}(X)$ in the sense of Coifman and Weiss.

    We now prove Lemma \ref{lipcmop6}. The idea of the proof of Lemma \ref{lipcmop6} is similar to the proof on $\mathbb R^n$ given by C. Fefferman \cite {FS}.
\begin{proof}[Proof of Lemma \ref{lipcmop6}]
For any fixed quasi-dyadic ball $Q=Q_{\alpha_0}^{k_0}$ in the sense of Auscher and Hyt\"onen, let $Q^*=2A_0Q$, the ball centered at the same center as $Q$ with the radius as $2A_0$ times of $Q.$ We denote  $f_Q=\frac{1}{\mu(Q)}\int_Q f(x)d\mu(x)$ and write $f=f_1+f_2+f_3,$ where $f_1=(f-f_{Q^*})\chi_{Q^*}, f_2=(f-f_{Q^*})\chi_{(Q^*)^c}$ and $f_3=f_{Q^*}$. Then $\langle f, \psi_\alpha^k\rangle=\langle f_1, \psi_\alpha^k\rangle+\langle f_2, \psi_\alpha^k\rangle$ since $\int_X\psi_\alpha^k(x)d\mu(x)=0$. This implies that
$\parallel f\parallel_{\mathrm{CMO}^p}\lesssim\parallel f_1\parallel_{\mathrm{CMO}^p}+\parallel f_2\parallel_{\mathrm{CMO}^p}$.

The estimate for $\parallel f_1\parallel_{\mathrm{CMO}^p}$ follows directly from the following
\begin{eqnarray*}
            \Big\{ {{1}\over{\mu(Q)^{{2\over p }- 1}}}
            \sum_{k\in\mathbb{Z}, \alpha\in\mathscr{Y}^{k},
                 Q_{\alpha}^{k} \subset Q}
            \big| \langle \psi_{\alpha}^{k}, f_1 \rangle
            \big|^2 \Big\}^{1/2}
            &\lesssim&{1\over\mu(Q)^{{1\over p} -\frac{1}{2} }}
                       \Big\{\int_{X} \big| f_1(x)\big|^2 d\mu(x) \Big\}^{1/2}\\
             &\lesssim&{1\over\mu(Q)^{{1\over p} -1 }}
                        \Big\{\frac{1}{\mu( Q)}\int_{Q^*}|f(x)-f_{Q^*}|^2 d\mu(x)\Big\rbrace^{\frac{1}{2}}\\
             &\lesssim& C \parallel f\parallel_{\mathcal C_{\frac{1}{p}-1}},
        \end{eqnarray*}
        where the doubling property on $\mu$ is used.
     To estimate $\parallel f_2\parallel_{\mathrm{CMO}^p(X)},$ observe that, by a result in \cite{HLW}, $\mu(Q_\alpha^k)^{-\frac{1}{2}}\psi_\alpha^k\in \GGs(\eta,\gamma)$ for any $\gamma>0,$
     particularly, we will take $\gamma>\omega(\frac{2}{p}-2).$ Thus,
     \begin{eqnarray*}|\langle f_2, \psi_\alpha^k\rangle|&\lesssim&\mu(Q_\alpha^k)^\frac{1}{2} \int_{(Q^*)^c}|f(x)-f_{Q^*}|{1\over V_{\delta^k}(x) + V(x,x_\alpha^k)}\Big({\delta^k\over \delta^{k} + d(x,x_\alpha^k)}\Big)^{\gamma}d\mu(x).
     \end{eqnarray*}
     Applying H\"older inequality together with the fact that
     $$\int_{X}{1\over V_{\delta^k}(x) + V(x,x_\alpha^k)}\Big({\delta^k\over \delta^{k} + d(x,x_\alpha^k)}\Big)^{\gamma}d\mu(x)\leq C$$ implies that
     \begin{eqnarray*}|\langle f_2, \psi_\alpha^k\rangle|^2&\lesssim&\mu(Q_\alpha^k) \int_{(Q^*)^c}|f(x)-f_{Q^*}|^2{1\over V_{\delta^k}(x) + V(x,x_\alpha^k)}\Big({\delta^k\over \delta^{k} + d(x,x_\alpha^k)}\Big)^{\gamma}d\mu(x).
     \end{eqnarray*}
     Next, as is easily seen, if $x\in (Q^*)^c$ and $Q_{\alpha}^{k} \subset Q$
     \begin{eqnarray*}
      {1\over V_{\delta^k}(x) + V(x,x_\alpha^k)}\Big({\delta^k\over \delta^{k} + d(x,x_\alpha^k)}\Big)^{\gamma}
      \lesssim {1\over V(x_{\alpha_0}^{k_0},x)}\Big({\delta^{k}\over \delta^{k_0} + d(x_{\alpha_0}^{k_0},x)}\Big)^{\gamma}.
     \end{eqnarray*}
     Therefore, we obtain
\begin{eqnarray*}
        &&\bigg\{ {1\over\mu(Q)^{{2\over p} - 1}}
            \sum_{k\in\mathbb{Z}, \alpha\in\mathscr{Y}^{k},
                 Q_{\alpha}^{k} \subset Q}
            \big| \langle \psi_{\alpha}^{k}, f_2 \rangle
            \big|^2 \bigg\}^{1/2}\\
     &\lesssim& \bigg\{ {1\over\mu(Q)^{{2\over p} -1}}\sum\limits_{k=k_0}^\infty
            \sum_{\alpha\in\mathscr{Y}^{k}:
                 Q_{\alpha}^{k} \subset Q} \mu(Q^k_\alpha)\int_{(Q^*)^c}|f(x)-f_{Q^*}|^2
                 {1\over V(x_{\alpha_0}^{k_0},x)}\Big({\delta^{k}\over \delta^{k_0} + d(x_0,x)}\Big)^{\gamma}
                 d\mu(x)\bigg\}^{1/2}\\
                 &\lesssim& \bigg\{\sum\limits_{k=k_0}^\infty {\delta^{\gamma(k-k_0)}\over\mu(Q)^{{2\over p} -2}}
                 \int_{(Q^*)^c}|f(x)-f_{Q^*}|^2
                 {1\over V(x_{\alpha_0}^{k_0},x)}\Big({\delta^{k_0}\over \delta^{k_0} + d(x_{\alpha_0}^{k_0},x)}\Big)^{\gamma}
                 d\mu(x)\bigg\}^{1/2}.        \end{eqnarray*}
     We claim that
     \begin{eqnarray*}
     \bigg\{{1\over\mu(Q)^{{2\over p} -2}}
                 \int_{(Q^*)^c}|f(x)-f_{Q^*}|^2
                 {1\over V(x_{\alpha_0}^{k_0},x)}\Big({\delta^{k_0}\over \delta^{k_0} + d(x_{\alpha_0}^{k_0},x)}\Big)^{\gamma}
                 d\mu(x)\bigg\}^{1/2}\lesssim ||f||_{{\mathcal C}_{\frac{1}{p}-1}},
     \end{eqnarray*}
     which gives the desired estimate for $\parallel f_2\parallel_{\mathrm{CMO}^p(X)}$ and hence the proof of Lemma \ref{lipcmop6} is concluded. To verify the claim, let $Q_j^{**}=\delta^{-j-1}Q^*$ for $j\in\mathbb{N}$, we have
\begin{eqnarray*}&&\bigg\{{1\over\mu(Q)^{{2\over p} -2}}
                 \int_{(Q^*)^c}|f(x)-f_{Q^*}|^2
                 {1\over V(x_{\alpha_0}^{k_0},x)}\Big({\delta^{k_0}\over \delta^{k_0} + d(x_{\alpha_0}^{k_0},x)}\Big)^{\gamma}
                 d\mu(x)\bigg\}^{1/2}\\
     &\lesssim&\bigg\{{1\over\mu(Q)^{{2\over p} -2}} \sum\limits_{j=0}^\infty\delta^{j\gamma}\int_{\{x:2A_0\delta^{k_0-j}\leq d(x,x_{\alpha_0}^{k_0}) <2A_0\delta^{k_0-j-1}\}}|f(x)-f_{Q_j^{**}}|^2 {1\over V(x_{\alpha_0}^{k_0},x)} d\mu(x)\bigg\}^{1/2}\\&+&\bigg\{{1\over\mu(Q)^{{2\over p} -2}} \sum\limits_{j=0}^\infty\delta^{j\gamma}\int_{\{x:2A_0\delta^{k_0-j}\leq d(x,x_{\alpha_0}^{k_0}) <2A_0\delta^{k_0-j-1}\}}|f_{Q_j^{**}}-f_{Q^*}|^2 {1\over V(x_{\alpha_0}^{k_0},x)} d\mu(x)\bigg\}^{1/2}
     \\
     &=:&I+II. \end{eqnarray*}For $I,$ by the doubling property on $\mu,$ we have
     \begin{eqnarray*}
     I     &\lesssim&\bigg\{{1\over\mu(Q)^{{2\over p} -2}} \sum\limits_{j=0}^\infty\delta^{j\gamma}{1\over  \mu(Q_j^{**})}\int_{Q_j^{**}}|f(x)-f_{Q_j^{**}}|^2                d\mu(x)\bigg\}^{1/2}
     \\
     &\lesssim&
            \Big\{           \sum\limits_{j=0}^\infty\delta^{j\gamma}{\mu(Q_j^{**})^{{2\over p} -2 }\over \mu(Q)^{{2\over p} -2} }
            \Big\}^{1/2}\parallel f\parallel_{C^{\frac{1}{p}-1}}\\
     &\lesssim&
            \Big\{           \sum\limits_{j=0}^\infty\delta^{j\gamma}\delta^{-j\omega\big(\frac{2}{p}-2\big)}
            \Big\}^{1/2}\parallel f\parallel_{C^{\frac{1}{p}-1}}
             \\&\lesssim &\parallel f\parallel_{C^{\frac{1}{p}-1}}
        \end{eqnarray*}since $\gamma>\omega(\frac{2}{p}-2)$.
      The estimate for $II$ follows directly from the doubling property on $\mu.$ Indeed,
        \begin{eqnarray*}
     II&\lesssim &\bigg\{{1\over\mu(Q)^{{2\over p} -2}} \sum\limits_{j=0}^\infty\delta^{j\gamma}|f_{Q_j^{**}}-f_{Q^*}|^2\frac{\mu(Q_j^{**})}{\mu(Q_{j-1}^{**})}
     \bigg\}^{1/2}
     \\
     &\lesssim&\bigg\{{1\over\mu(Q)^{{2\over p} -2}}\sum\limits_{j=0}^\infty\delta^{j\gamma} {1\over  \mu(Q^{*})}\int_{Q^{*}}|f(x)-f_{Q_j^{**}}|^2d\mu(x)\bigg\}^{1/2}
     \\
     &\lesssim&\bigg\{\sum\limits_{j=0}^\infty\delta^{j\gamma}{\mu(Q_j^{**})^{{2\over p} -1}\over\mu(Q)^{{2\over p} -2}} {1\over  \mu(Q^{*})}{1\over  \mu(Q_j^{**})^{{2\over p} -1}}\int_{Q_j^{**}}|f(x)-f_{Q_j^{**}}|^2d\mu(x)\bigg\}^{1/2}
                    \\
     &\lesssim&            \Big\{           \sum\limits_{j=0}^\infty\delta^{j\gamma}\delta^{-j\omega\big(\frac{2}{p}-2\big)}
            \Big\}^{1/2}\parallel f\parallel_{C^{\frac{1}{p}-1}}
             \\&\lesssim &\parallel f\parallel_{C^{\frac{1}{p}-1}},
        \end{eqnarray*}where the fact that $\gamma>\omega(\frac{2}{p}-2)$ is used. The proof of the claim is complete and hence the proof of Lemma \ref{lipcmop6} is concluded.
\end{proof}

\section{Criterion of the boundedness for singular integrals on the Hardy spaces}\label{sec:proof of theorem 1.9}

\subsection{Molecule theory on spaces of homogeneous type}\label{subsec:Molecule theory on spaces of homogeneous type}
As mentioned in Section 1, we develop the theory of molecule on spaces of homogeneous type in the sense of Coifman and Weiss. Let $(X,d,\mu)$ be a space of homogeneous type. We define the molecule which depends only on the measure $\mu$. Since we do not have any conditions on the measure other than the doubling condition, we applying a stopping time argument in proving the molecule theory, which is new comparing to all the previous related versions of molecule theory.

\begin{definition}\label{def-molecule}
Suppose $ {\omega\over \omega+\eta} <p\leq1$. A function $m(x)\in L^2(X)$ is said to be a $(p,2,\epsilon)$ molecule if $\epsilon>0$, $ {\omega\over \omega+\eta-\epsilon} <p\leq1$  and
\begin{eqnarray}\label{molecule}
\Big( \int_X m(x)^2 d\mu(x) \Big)\Big( \int_X m(x)^2 V(x,x_0)^{1+{2\eta-2\epsilon\over \omega}} d\mu(x)\Big)^{({\omega+2\eta-2\epsilon\over \omega}{p\over 2-p}-1)^{-1}}\leq 1,
\end{eqnarray}
where $x_0$ is a fixed point in $X$,  $\omega$ is the upper dimension of $\mu$.
\end{definition}
Note that the fact that $ {\omega\over \omega+\eta-\epsilon} <p$ implies ${\omega+2\eta-2\epsilon\over \omega}{p\over 2-p}-1>0$ and moreover, the quasi metric $d$ is not used in the definition of molecules.
Next we show that each $(p,2,\epsilon)$ molecule $m(x)$ belongs to $H^p(X).$
\begin{theorem}\label{moleculeinHp}
Suppose that $m$ is an $(p,2,\epsilon)$ molecule. Then $m\in H^p(X)$ and moreover
$$\|m\|_{H^p}\leq C,$$
where the constant $C$ is independent of $m.$
\end{theorem}

\begin{proof}
The basic idea of the proof is to decompose $m$ into a sum of atoms. For this purpose, we first set $\sigma=\|m\|_{L^2}^{-\alpha}$, where $\alpha={2p\over (2-p) \omega}$.

First, we point out that $B(x_0,2^{i+1}\sigma) \rightarrow X$ when $i$ tends to $+\infty$, and that $\mu(X)=+\infty$. Thus, there exists an integer $i_0$ such that
\begin{eqnarray}\label{molecule i0}
\mu(B(x_0,2^{i_0+1}\sigma))>\sigma^{\omega} \textup{\ \ \   and\ \ \   }
\mu(B(x_0,2^{i_0}\sigma))\leq \sigma^{\omega}.
\end{eqnarray}

Set
$$\chi_0=B(x_0, 2^{i_0}\sigma)=\{x\in X: d(x,x_0) < 2^{i_0}\sigma \} $$
and for $i\geq 1,$
$$
  \chi_i=B(x_0,2^i2^{i_0}\sigma)\backslash B(x_0,2^{i-1}2^{i_0}\sigma)=\{x\in X:  2^{i-1} 2^{i_0}\sigma \leq d(x,x_0) < 2^i 2^{i_0}\sigma\}.
$$
Let $\chi_i(x)$ be the characteristic function on $\chi_i, i=0,1,2,...$.

We claim that there exists an integer $j_1\geq 1$ such that
$$
  \mu\Big(\bigcup_{\ell=1}^{j_1} \chi_\ell \Big) > \mu(\chi_0)
$$
and
$$
  \mu\Big(\bigcup_{\ell=1}^{j_1-1} \chi_\ell \Big) \leq \mu(\chi_0). \ \ \ \textup{If $j_1=1$, then this does not apply}.
$$
To verify the claim, suppose that such $j_1$ does not exist. Then for every integer $j>1$, we should have $\mu\Big(\bigcup_{\ell=1}^{j} \chi_\ell \Big) \leq \mu(\chi_0)$. This implies that
$\mu(X)=\mu\Big(\lim_{j\rightarrow\infty}\bigcup_{\ell=1}^{j} \chi_\ell \Big)=\lim_{j\rightarrow\infty} \mu\Big(\bigcup_{\ell=1}^{j} \chi_\ell \Big)   \leq
\mu(\chi_0)$ and it contradicts with the fact that $\mu(X)=\infty.$
Applying the same stopping time argument yields that there exists a sequence $\{j_k\}_k$ such that $j_{k}>j_{k-1}$ and
$$
  \mu\Big(\bigcup_{\ell=j_{k}+1}^{j_{k+1}} \chi_\ell \Big) > \mu(B(x_0,2^{j_k}2^{i_0}\sigma))
$$
and
$$
  \mu\Big(\bigcup_{\ell=j_{k}+1}^{j_{k+1}-1} \chi_\ell \Big) \leq \mu(B(x_0,2^{j_k}2^{i_0}\sigma)). \ \ \ \textup{If $j_{k+1}=j_k+1$, then this does not apply}.
$$

Observe that
\begin{eqnarray}\label{molecule e1}
 \mu(B(x_0,2^{j_{k+1}}2^{i_0}\sigma))=\mu(B(x_0,2^{j_k}2^{i_0}\sigma))+ \mu\Big(\bigcup_{\ell=j_{k}+1}^{j_{k+1}} \chi_\ell \Big)  \geq 2\mu(B(x_0,2^{j_k}2^{i_0}\sigma))
\end{eqnarray}
for each integer $k\geq0$. Here we set $j_0=0$.

Applying (\ref{molecule e1}) and induction yields
\begin{eqnarray}\label{molecule e3}
\mu(B(x_0,2^{j_k}2^{i_0} \sigma))\geq  2^k \mu(B(x_0,2^{i_0} \sigma))
\geq  2^k C_\mu^{-1}\mu(B(x_0,2^{i_0+1} \sigma))
\geq 2^k C_\mu^{-1} \sigma^\omega,
\end{eqnarray}
where the second inequality follows from the doubling condition of the measure $\mu$ and the last inequality follows from the definition of the integer $i_0$, see (\ref{molecule i0}).

We point out that for each integer $k\geq1$, if $j_k=j_{k-1}+1$, then we directly obtain that  $\mu( B(x_0, 2^{j_k}2^{i_0}\sigma ) ) \leq C_\mu \mu( B(x_0, 2^{j_{k-1}}2^{i_0}\sigma ) )$ from the doubling property of the measure $\mu$. While if $j_k>j_{k-1}+1$, then
\begin{eqnarray*}
\mu( B(x_0, 2^{j_k}2^{i_0}\sigma ) ) &\leq& C_\mu \mu( B(x_0, 2^{j_k-1}2^{i_0}\sigma ) ) \nonumber\\
&=& C_\mu\mu( B(x_0, 2^{j_k-1}2^{i_0}\sigma )\backslash B(x_0, 2^{j_{k-1}}2^{i_0}\sigma ) )+C_\mu\mu( B(x_0, 2^{j_{k-1}}2^{i_0}\sigma ) ).
\end{eqnarray*}
Note that
\begin{eqnarray*}
\mu( B(x_0, 2^{j_k-1}2^{i_0}\sigma )\backslash B(x_0, 2^{j_{k-1}}2^{i_0}\sigma ) )= \mu\Big(\bigcup_{\ell= j_{k-1}+1}^{j_k-1}\chi_\ell\Big)
\leq \mu( B(x_0, 2^{j_{k-1}}2^{i_0}\sigma ) ),
\end{eqnarray*}
which, together with the above estimate for the case $j_k=j_{k-1}+1$, yields
\begin{eqnarray}
\mu( B(x_0, 2^{j_k}2^{i_0}\sigma ) )
&\leq& 2C_\mu \mu( B(x_0, 2^{j_{k-1}}2^{i_0}\sigma ) )\label{molecule e4}
\end{eqnarray}
for each integer $k\geq1$.

We also point out that, by (\ref{molecule e4}), we obtain
\begin{eqnarray*}
\mu( B(x_0, 2^{j_{k+1}} 2^{i_0}\sigma ) )
&\leq& 2C_\mu \mu( B(x_0, 2^{j_{k}}2^{i_0}\sigma )),
\end{eqnarray*}
which together with the following estimates
\begin{eqnarray*}
\mu( B(x_0, 2^{j_{k}}2^{i_0}\sigma ))
&\leq&  \mu( B(x_0, 2^{j_{k-1}}2^{i_0}\sigma ))+\mu\Big(\bigcup_{\ell= j_{k-1}+1 }^{j_k}\chi_{\ell}\Big)\\
&\leq&  2\mu\Big(\bigcup_{\ell= j_{k-1}+1 }^{j_k}\chi_{\ell}\Big)
\end{eqnarray*}
gives
\begin{eqnarray}\label{molecule e7}
\mu( B(x_0, 2^{j_{k+1}} 2^{i_0}\sigma ) )
&\leq& 4C_\mu\mu\Big(\bigcup_{\ell= j_{k-1}+1 }^{j_k}\chi_{\ell}\Big).
\end{eqnarray}

We now set
$$\widetilde{\chi}_0(x):=\chi_0(x),\hskip1cm
 \widetilde{\chi}_{j_k}(x):=\sum_{\ell=j_{k-1}+1}^{j_{k}} \chi_\ell(x).$$
for integer $k\geq1$, and
\begin{eqnarray}\label{molecule e2}
 m_k(x):= m(x)\widetilde{\chi}_{j_k}(x)- {1\over \int \widetilde{\chi}_{j_k}(z)d\mu(z)} \int_{\widetilde{\chi}_{j_k}}m(y)d\mu(y) \widetilde{\chi}_{j_k}(x)
\end{eqnarray}
for each integer $k\geq0$.

Decompose $m$ by
$$m(x):=\sum_{k=0}^{\infty}m_k(x)+\sum_{k=0}^\infty \overline{m}_k \widetilde{\widetilde{\chi}}_{j_k}(x),  $$
where $\overline{m}_k=\int m(x)\widetilde{\chi}_{j_k}(x)d\mu(x)$ and $\widetilde{\widetilde{\chi}}_{j_k}(x)= {\widetilde{\chi}_{j_k}(x) \over \int \widetilde{\chi}_{j_k}(y)d\mu(y)}$.

To see that $\sum_{k=0}^{\infty}m_k(x)$ gives an atomic decomposition, we will show that $m_k$ are $(p,2)$ atoms due to multiplication of certain constant. Note that $m_k$ is supported in $\widetilde{\chi}_{j_k}=\bigcup_{\ell=j_{k-1}+1}^{j_{k}} \chi_\ell$, and that $\int m_k(x)d\mu(x)=0 $. Therefore, we only need to estimate the $L^2$ norm of $m_k$. First, we have
\begin{eqnarray*}
\|m_0\|_{L^2}&\leq& \Big( \int_{{\chi}_0} |m(x)|^2 d\mu(x) \Big)^{1/2}+\Big( \int_{{\chi}_0} \Big|{1\over \int {\chi}_0(z)d\mu(z)} \int_{\chi_0}m(y)d\mu(y) \widetilde{\chi}_0(x)\Big|^2 d\mu(x) \Big)^{1/2}\\
&\leq& 2\Big( \int_{{\chi}_0} |m(x)|^2 d\mu(x) \Big)^{1/2}\\
&\leq& 2\|m\|_{L^2}\\
&=& 2\sigma^{-{1\over \alpha}}\\
&\leq& 2 \mu( B(x_0,2^{i_0}\sigma) )^{-{1\over \alpha\omega}}\\
&=& 2\mu({\chi}_0)^{{1\over2}-{1\over p}},
\end{eqnarray*}
where in the last inequality we use the fact in (\ref{molecule i0}), namely that $\mu(B(x_0,2^{i_0}\sigma))\leq \sigma^{\omega}.$

Thus, $ 2^{-1}m_0(x) $ is a $(p,2)$ atom.

Similarly, for each $k\geq1$,
\begin{eqnarray*}
\|m_k\|_{L^2}
&\leq& 2\Big( \int_{\widetilde{\chi}_{j_k}} |m(x)|^2 d\mu(x) \Big)^{1/2}\\
&=& 2\Big( \int_{\widetilde{\chi}_{j_k}} |m(x)|^2 V(x,x_0)^{1+{2\eta-2\epsilon\over\omega}} V(x,x_0)^{-(1+{2\eta-2\epsilon\over\omega})}  d\mu(x) \Big)^{1/2}\\
&\leq& 2 V(x_0,2^{j_{k-1}}2^{i_0}\sigma)^{-{1\over2}-{\eta-\epsilon\over\omega}}  \Big( \int_{\widetilde{\chi}_{j_k}} |m(x)|^2  V(x,x_0)^{1+{2\eta-2\epsilon\over\omega}} d\mu(x) \Big)^{1/2}\\
&\leq& 2  V(x_0,2^{j_{k-1}}2^{i_0}\sigma)^{-{1\over2}-{\eta-\epsilon\over\omega}}   \|m\|_{L^2}^{-({\omega+2\eta-2\epsilon\over \omega}{p\over 2-p}-1)}\\
&=&    2  V(x_0,2^{j_{k-1}}2^{i_0}\sigma)^{-{1\over2}-{\eta-\epsilon\over\omega}}  \sigma^{{1\over \alpha} ({\omega+2\eta-2\epsilon\over \omega}{p\over 2-p}-1) }.
\end{eqnarray*}
Applying the estimates given in (\ref{molecule e3}) yields
\begin{eqnarray*}
\sigma^{{1\over \alpha} ({\omega+2\eta-2\epsilon\over \omega}{p\over 2-p}-1)}
\leq(C_\mu2^{-k}V(x_0,2^{j_{k}}2^{i_0}\sigma)^{{1\over \omega\alpha} ({\omega+2\eta-2\epsilon\over \omega}{p\over 2-p}-1)},
\end{eqnarray*}
which, together with the estimates given in (\ref{molecule e4}), namely that $$V(x_0,2^{j_{k-1}}2^{i_0}\sigma)^{-{1\over2}-{\eta-\epsilon\over\omega}}\leq (2C_\mu)^{{1\over2}+{\eta-\epsilon\over\omega}}V(x_0,2^{j_{k}}2^{i_0}\sigma)^{-{1\over2}-{\eta-\epsilon\over\omega}}$$ implies
\begin{eqnarray}
\|m_k\|_{L^2}
&\leq&  2  V(x_0,2^{j_{k-1}}2^{i_0}\sigma)^{-{1\over2}-{\eta-\epsilon\over\omega}}  \sigma^{{1\over \alpha} ({\omega+2\eta-2\epsilon\over \omega}{p\over 2-p}-1) }\nonumber\\
&\leq& 2\cdot 2^{ -{1\over \omega\alpha} ({\omega+2\eta-2\epsilon\over \omega}{p\over 2-p}-1)  k}     (2C_\mu)^{{1\over2}+{\eta-\epsilon\over\omega}} C_\mu^{ {1\over \omega\alpha} ({\omega+2\eta-2\epsilon\over \omega}{p\over 2-p}-1)  }      V(x_0,2^{j_{k}}2^{i_0}\sigma)^{({1\over 2}-{1\over p})}.\label{molecule e5}
\end{eqnarray}
Thus, $2^{-1}\cdot 2^{ {1\over \omega\alpha} ({\omega+2\eta-2\epsilon\over \omega}{p\over 2-p}-1)  k}     (2C_\mu)^{-{1\over2}-{\eta-\epsilon\over\omega}} C_\mu^{ -{1\over \omega\alpha} ({\omega+2\eta-2\epsilon\over \omega}{p\over 2-p}-1)  }  m_k   $ are $(p,2)$ atoms.

Moreover, $\sum_k 2^{ -{1\over \omega\alpha} ({\omega+2\eta-2\epsilon\over \omega}{p\over 2-p}-1)kp}<\infty.$
As a consequence,  $\sum_{k=0}^{\infty}m_k(x)$ gives the desired atomic decomposition and hence, by Theorem \ref{thm 1.1}, belongs to $H^p(X)$ with the norm not larger than the constant $C$, which depends only on $p, \omega, \eta, \epsilon$ and $C_\mu$.

It remains to show that $\sum_{k=0}^\infty \overline{m}_k \widetilde{\widetilde{\chi}}_{j_k}(x)$ also gives an atomic decomposition. To see this, let $N_{k'}=\sum_{k=k'}^\infty \overline{m}_k$. Note that $\sum_{k=0}^\infty \overline{m}_k= \int m(x)d\mu(x)=0.$ Summing up by parts implies that
$$
  \sum_{k=0}^\infty \overline{m}_k\widetilde{\widetilde{\chi}}_k(x)=\sum_{k'=0}^\infty (N_{k'}-N_{k'+1})\widetilde{\widetilde{\chi}}_{k'}(x)=\sum_{k'=0}^\infty N_{k'+1}(\widetilde{\widetilde{\chi}}_{k'+1}(x)-\widetilde{\widetilde{\chi}}_{k'}(x)).
$$
Observe that the support of $\widetilde{\widetilde{\chi}}_{k'+1}(x)-\widetilde{\widetilde{\chi}}_{k'}(x)$ lies within $ B(x_0,2^{k'+1}\sigma) $ and $$\int_X (\widetilde{\widetilde{\chi}}_{k'+1}(x)-\widetilde{\widetilde{\chi}}_{k'}(x)) d\mu(x)=0$$ since $\int\widetilde{\widetilde{\chi}}_{k'}(x)d\mu(x)=1 $ for all $k'$. And we also have
$$
|\widetilde{\widetilde{\chi}}_{k'+1}(x)-\widetilde{\widetilde{\chi}}_{k'}(x) | \leq
{1 \over \int \widetilde{\chi}_{k'+1}(y)d\mu(y)}+{1 \over \int \widetilde{\chi}_{k'}(y)d\mu(y)}\leq {2 \over \int \widetilde{\chi}_{k'}(y)d\mu(y)}={2\over \mu(\widetilde{\chi}_{k'})}.
$$
Now applying (\ref{molecule e7}),  we obtain that
\begin{eqnarray}\label{molecule e6}
|\widetilde{\widetilde{\chi}}_{k'+1}(x)-\widetilde{\widetilde{\chi}}_{k'}(x) | \leq {8C_\mu\over \mu( B(x_0, 2^{j_{k'+1}}2^{i_0}\sigma ) )}.
\end{eqnarray}
Applying the H\"older inequality and the estimates in (\ref{molecule e5}), we obtain that
\begin{eqnarray*}
|N_{k'+1}|&\leq& \sum_{k=k'+1}^\infty \int |m(x)\widetilde{\chi}_{j_{k}}(x)|d\mu(x)\\
&\leq& C \sum_{k=k'+1}^\infty \Big(\int_{\widetilde{\chi}_{j_{k}}} |m(x)|^2d\mu(x)\Big)^{1/2}
\mu(\widetilde{\chi}_{j_{k}})^{1/2}\\
&\leq& C \sum_{k=k'+1}^\infty   2 \cdot 2^{ -{1\over \omega\alpha} ({\omega+2\eta-2\epsilon\over \omega}{p\over 2-p}-1)  k}     (2C_\mu)^{{1\over2}+{\eta-\epsilon\over\omega}} C_\mu^{ {1\over \omega\alpha} ({\omega+2\eta-2\epsilon\over \omega}{p\over 2-p}-1)  }      \\
&&\times V(x_0,2^{j_{k}}2^{i_0}\sigma)^{({1\over 2}-{1\over p})}\mu( B(x_0, 2^{j_k}2^{i_0}\sigma ))^{1/2}\\
&\leq& C 2 \cdot 2^{ -{1\over \omega\alpha} ({\omega+2\eta-2\epsilon\over \omega}{p\over 2-p}-1)  (k'+1)}     (2C_\mu)^{{1\over2}+{\eta-\epsilon\over\omega}} C_\mu^{ {1\over \omega\alpha} ({\omega+2\eta-2\epsilon\over \omega}{p\over 2-p}-1)  }  \\ &&\times \mu( B(x_0, 2^{j_{k'+1}}2^{i_0}\sigma ) )^{1-{1\over p}}.
\end{eqnarray*}
The estimate above and the size estimate of $\widetilde{\widetilde{\chi}}_{k'+1}(x)
-\widetilde{\widetilde{\chi}}_{k'}(x)$ in (\ref{molecule e6}) imply
\begin{eqnarray*}
|N_{k'+1}(\widetilde{\widetilde{\chi}}_{k'+1}(x)
-\widetilde{\widetilde{\chi}}_{k'}(x))|
&\leq& C 2\cdot 2^{ -{1\over \omega\alpha} ({\omega+2\eta-2\epsilon\over \omega}{p\over 2-p}-1)  (k'+1)}     (2C_\mu)^{{1\over2}+{\eta-\epsilon\over\omega}} C_\mu^{ {1\over \omega\alpha} ({\omega+2\eta-2\epsilon\over \omega}{p\over 2-p}-1)  }  \\
&&\times {6C_\mu\over \mu( B(x_0, 2^{j_{k'+1}}\sigma ) )}     \mu( B(x_0, 2^{j_{k'+1}}2^{i_0}\sigma ) )^{1-{1\over p}}\\
&\leq& C  2^{ -{1\over \omega\alpha} ({\omega+2\eta-2\epsilon\over \omega}{p\over 2-p}-1)  (k'+1)} \mu( B(x_0, 2^{j_{k'+1}}2^{i_0}\sigma ) )^{-{1\over p}}.
\end{eqnarray*}
Therefore, we can rewrite $N_{k'+1}(\widetilde{\widetilde{\chi}}_{k'+1}(x)
-\widetilde{\widetilde{\chi}}_{k'}(x))$ as
$$ N_{k'+1}(\widetilde{\widetilde{\chi}}_{k'+1}(x)
-\widetilde{\widetilde{\chi}}_{k'}(x))= \alpha_{k'}\beta_{k'}(x), $$
where $\alpha_{k'}=C 2^{ -{1\over \omega\alpha} ({\omega+2\eta-2\epsilon\over \omega}{p\over 2-p}-1)  (k'+1)}$
and $\beta_{k'}(x)$ are $(p, 2)$ atoms. Hence, by Theorem \ref{thm 1.1},
$ \sum_{k=0}^\infty \overline{m}_k\widetilde{\widetilde{\chi}}_k(x)$ belongs to $H^p(X)$ with the $H^p(X)$ norm does not exceed $C$.
The proof of Theorem \ref{moleculeinHp} is concluded.
\end{proof}

\subsection{The Proof of Theorem 1.3}\label{subsec:proof of theorem 1.3}

\begin{proof}[Proof of Theorem \ref{thm T1Hp}]
Let ${\omega\over \omega+\eta}<p\leq1$ and $\epsilon>0$ such that
${\omega\over \omega+\eta-\epsilon}<p\leq1.$ Suppose that $T$ is a singular integral operator with kernel $K(x,y)$ satisfying the estimate \eqref{size of C-Z-S-I-O} and \eqref{y smooth of C-Z-S-I-O}, and $T$ is bounded on $L^2(X)$. Note that $L^2(X)\cap H^p(X)$ is dense in $H^p(X)$ and if $f\in L^2(X)\cap H^p(X)$ then $f$ has an atomic decomposition $f=\sum_j \lambda_j a_j$ where the series converges in both $L^2(X)$ and $H^p(X).$ Therefore,
to show that $T$ extends to be a bounded operator on $H^{p}(X)$, it suffices to verify that for each $(p,2)$-atom $a$, $ m=T(a) $ is an $(p,2,\epsilon)$-molecule up to a multiplication of a constant $C$.
Suppose that $a$ is an $(p,2)$ atom with the support $B(x_0,r).$ We write
\begin{eqnarray*}
&&\Big( \int_X m(x)^2 d\mu(x) \Big)\Big( \int_X m(x)^2 V(x_0,x)^{1+{2\eta-2\epsilon\over \omega}} d\mu(x)\Big)^{({\omega+2\eta-2\epsilon\over \omega}{p\over 2-p}-1)^{-1}}\\
&&\leq \Big( \int_X m(x)^2 d\mu(x) \Big)\Big( \int_{d(x_0,x)\leq2r} m(x)^2V(x_0,x)^{1+{2\eta-2\epsilon\over \omega}} d\mu(x)\Big)^{({\omega+2\eta-2\epsilon\over \omega}{p\over 2-p}-1)^{-1}}\\
&& \hskip1cm+ \Big( \int_X m(x)^2 d\mu(x) \Big)\Big( \int_{d(x_0,x)>2r} m(x)^2 V(x_0,x)^{1+{2\eta-2\epsilon\over \omega}} d\mu(x)\Big)^{({\omega+2\eta-2\epsilon\over \omega}{p\over 2-p}-1)^{-1}}\\
&&:=I+II.
\end{eqnarray*}
Note that, by the $L^2$ boundedness of $T$ and the size condition on $a, \|m\|_{L^2}^2\leq V(x_0,r)^{(1-{2\over p})}.$ As for $I$, applying the doubling property on $\mu$ implies that
\begin{eqnarray*}
I&\leq& CV(x_0,r)^{(1-{2\over p})}V(x_0,2r)^{(1+{2\eta-2\epsilon\over \omega})({\omega+2\eta-2\epsilon\over \omega}{p\over 2-p}-1)^{-1}}
 \Big( \int m(x)^2  d\mu(x)\Big)^{({\omega+2\eta-2\epsilon\over \omega}{p\over 2-p}-1)^{-1}}\\
&\leq&  CV(x_0,r)^{(1-{p\over2})}V(x_0,r)^{(1+{2\eta-2\epsilon\over \omega})({\omega+2\eta-2\epsilon\over \omega}{p\over 2-p}-1)^{-1}}
 V(x_0,r)^{(1-{2\over p})({\omega+2\eta-2\epsilon\over \omega}{p\over 2-p}-1)^{-1}}\\
&\leq& C.
\end{eqnarray*}
To estimate $II,$ observe that if $d(x_0,x)>2r$, by the cancellation condition on $a,$ we have
\begin{eqnarray*}
|m(x)|&=& \int[K(x,y)-K(x,x_0)]a(y)d\mu(y)\leq C\int {1\over V(x_0,x)} \Big( {d(x_0,y)\over d(x_0,x)} \Big)^\eta |a(y)|d\mu(y)\\
&\leq& C {1\over V(x_0,x)} \Big( {r\over d(x_0,x)} \Big)^\eta \mu(B(x_0,r))^{1-{1\over p}} .
\end{eqnarray*}
This tgether with the doubling property on $\mu$ gives
\begin{eqnarray*}
II&\leq&
 C\mu(B(x_0,r))^{(1-{2\over p})}\Big( \int_{d(x_0,x)>2r} m(x)^2 V(x_0,x)^{1+{2\eta-2\epsilon\over \omega}}  d\mu(x)\Big)^{({\omega+2\eta-2\epsilon\over \omega}{p\over 2-p}-1)^{-1}}\\
&\leq& C\mu(B(x_0,r))^{(1-{2\over p})}\\
&&\times\Big( \int_{d(x_0,x)>2r}  {1\over V(x_0,x)^2}  \Big( {r\over d(x_0,x)} \Big)^{2\eta} V(x,x_0)^{1+{2\eta-2\epsilon\over \omega}} \mu(B(x_0,r))^{2-{2\over p}} d\mu(x)\Big)^{({\omega+2\eta-2\epsilon\over \omega}{p\over 2-p}-1)^{-1}}\\
&\leq& C \mu(B(x_0,r))^{(1-{2\over p})}\\
&&\times\Big( r^{2\eta} \int_{d(x_0,x)>2r}  {V(x_0,x)^{{2\eta-2\epsilon\over \omega}}\over V(x_0,x)}  \Big( {1\over d(x,x_0)} \Big)^{2\eta} \mu(B(x_0,r))^{2-{2\over p}}  d\mu(x)\Big)^{({\omega+2\eta-2\epsilon\over \omega}{p\over 2-p}-1)^{-1}}\\
&\leq& C \mu(B(x_0,r))^{(1-{2\over p})}\Big( r^{2\eta} \sum_{j\geq 0}\int_{2^{j+1}r\leq d(x,x_0)<2^{j+2}r} {V(x_0, 2^{j+2}r)^{{2\eta-2\epsilon\over \omega}}\over V(x_0,2^{j+1}r)}\\
&&\times  \Big( {1\over 2^{j+1}r} \Big)^{2\eta} \mu(B(x_0,r))^{2-{2\over p}}  d\mu(x)\Big)^{({\omega+2\eta-2\epsilon\over \omega}{p\over 2-p}-1)^{-1}}\\
&\leq& C \mu(B(x_0,r))^{(1-{2\over p})}\Big( \mu(B(x_0,r))^{\frac{2\eta-2\epsilon}{\omega}+2-{2\over p}}  \Big)^{({\omega+2\eta-2\epsilon\over \omega}{p\over 2-p}-1)^{-1}}\\
&\leq& C.
\end{eqnarray*}
Finally, by the fact that $T^*(1)=0$, we obtain that $ \int m(x) d\mu(x)=\int T(a)(x)d\mu(x)=0$ and hence $m$ is the multiple of an $(p,2,\epsilon)$ molecule. The proof of the sufficient implication of Theorem \ref{thm T1Hp} then follows from Theorem \ref{moleculeinHp}.

We now prove that if $T$ is bounded on $L^2(X)$ and on $H^p(X)$ then $\int T(f)(x)d\mu(x)=0$ for $f\in L^2(X)\cap H^p(X).$ As mentioned in section 1, this follows from the following general result.
\begin{prop} \label{HpinLp}
If $f\in L^2(X)\cap H^p(X), \frac{\omega}{\omega+\eta}<p\leq 1,$ then there exists a constant $C$ independent of the $L^2(X)$ norm of $f$ such that
\begin{eqnarray}
\|f\|_p\leq C\|f\|_{H^p}.
\end{eqnarray}
\end{prop}
As a consequence of this proposition, if $f\in L^2(X)\cap H^p(X),$ then $f\in L^1(X)$ and $\int f(x) d\mu(x)=0.$ Indeed, if $f\in L^2(X)\cap H^p(X),$ by Proposition \ref{HpinLp}, then $f\in L^p(X)\cap L^2(X).$ Hence, by interpolation, $f\in L^1(X).$ To see the integral of $f$ is zero, applying the wavelet expansion as in Theorem \ref{thm 1.1}, $f(x)= \sum_{k\in\mathbb{Z}}\sum_{\alpha \in \mathscr{Y}^k}\langle f,\psi_{\alpha}^k \rangle \psi_{\alpha}^k(x)$ where the series converges in both $L^2(X)$ and $H^p(X).$ Let $E_n(k,\alpha)$ be a finite set of $k\in\mathbb{Z}$ and $\alpha \in \mathscr{Y}^k$ and $E_n(k,\alpha)$ tends to the whole set $\lbrace (k,\alpha): k\in\mathbb{Z}, \alpha \in \mathscr{Y}^k\rbrace.$ Therefore, $\sum_{E_n^c(k,\alpha)} \langle f,\psi_{\alpha}^k \rangle \psi_{\alpha}^k(x)$ converges to zero as $n$ tends to infinity in both $L^2(X)$ and $H^p(X).$ We obtain that
\begin{eqnarray*}
|\int f(x)d\mu(x)|&\leq& |\int\sum_{E_n(k,\alpha)} \langle f,\psi_{\alpha}^k \rangle \psi_{\alpha}^k(x)d\mu(x)|+
|\int\sum_{E_n^c(k,\alpha)} \langle f,\psi_{\alpha}^k \rangle \psi_{\alpha}^k(x)d\mu(x)|\\
&\leq&|\int \sum_{E_n^c(k,\alpha)} \langle f,\psi_{\alpha}^k \rangle \psi_{\alpha}^k(x)d\mu(x)|\\
&\leq& C\|\sum_{E_n^c(k,\alpha)} \langle f,\psi_{\alpha}^k \rangle \psi_{\alpha}^k(x)\|_{H^p}
+C\|\sum_{E_n^c(k,\alpha)} \langle f,\psi_{\alpha}^k \rangle \psi_{\alpha}^k(x)\|_2,
\end{eqnarray*}
where the first inequality follows from the fact that $\int\sum_{E_n(k,\alpha)} \langle f,\psi_{\alpha}^k \rangle \psi_{\alpha}^k(x)d\mu(x)=0$ since the cancellation property on the wavelet $\psi_{\alpha}^k(x).$

Letting $n$ tend to infinity gives the desired result.

Now assuming Proposition \ref{HpinLp} for the moment, if $T$ is bounded on both $L^2(X)$ and $H^p(X)$ and $f\in L^2(X)\cap H^p(X),$ then $Tf \in L^2(X)\cap H^p(X)$ and hence $\int Tf(x) d\mu(x)=0.$ The necessary implication of Theorem \ref{thm T1Hp} is concluded.

It remains to show Proposition \ref{HpinLp}. The key idea of the proof is to apply the method of atomic decomposition for subspace $L^2(X)\cap H^p(X).$ More precisely, if $f\in L^2(X)\cap H^p(X),$ as in the proof of Theorem \ref{thm 1.1}, we set
\[\Omega_l=\left\{ x\in X: {\widetilde{S}}(f)(x)>2^l\right\},\]
\[
B_l=\left\{ Q_\alpha^k:\mu(Q_\alpha^k\cap\Omega_l)>{1\over2}\mu(Q_\alpha^k)
\text{ and }
\mu(Q_\alpha^k \cap\Omega_{l+1})\leq {1\over2}\mu(Q_\alpha^k)\right\}
\]
and
\[\widetilde{\Omega}_l=\left\{ x\in {X}:  M(\chi_{\Omega_l})(x)>1/2\right\},\]
where $M$ is the Hardy--Littlewood maximal function on $X$ and hence, $\mu(\widetilde{\Omega}_l)\leq C\mu({\Omega}_l).$

Applying Proposition \ref{ new Caderon-type reproducing formula}, we write
\[
f(x)=\sum_{l}\sum_{ Q_\alpha^{k+N}\in B_l}  \mu(Q_\alpha^{k+N})D_k(x,x_\alpha^{k+N}) {\widetilde D}_k(g)(x_\alpha^{k+N}),
\]
where the series converges in both $L^2(X)$ and $H^p(X).$ Thus, for $\frac{{\omega}{\omega+\eta}}<p\leq 1,$
\[
\|f(x)\|_p^p\leq\sum_{l}\|\sum_{ Q_\alpha^{k+N}\in B_l}  \mu(Q_\alpha^{k+N})D_k(x,x_\alpha^{k+N}) {\widetilde D}_k(g)(x_\alpha^{k+N})\|_p^p.
\]
Note that if $Q_\alpha^{k+N}\in B_l$ then $Q_\alpha^{k+N}\subset \widetilde{\Omega}_l.$ Therefore, $\sum\limits_{ Q_\alpha^{k+N}\in B_l}  \mu(Q_\alpha^{k+N})D_k(x,x_\alpha^{k+N}) {\widetilde D}_k(g)(x_\alpha^{k+N})$ is supported
in $\widetilde{\Omega}_l.$ Applying H\"older inequality implies that
\[
\|f(x)\|_p^p\leq\sum_{l} \mu(\widetilde{\Omega}_l)^{1-\frac{p}{2}}\|\sum_{ Q_\alpha^{k+N}\in B_l}  \mu(Q_\alpha^{k+N})D_k(x,x_\alpha^{k+N}) {\widetilde D}_k(g)(x_\alpha^{k+N})\|_2^p.
\]
As in the proof of Theorem \ref{thm 1.1}, we have
\[
 \|\sum_{ Q_\alpha^{k+N}\in B_l}  \mu(Q_\alpha^{k+N})D_k(x,x_\alpha^{k+N}) {\widetilde D}_k(g)(x_\alpha^{k+N})\|_2\leq C2^l\mu(\widetilde{\Omega}_l)^{\frac{1}{2}},
\]
which gives
\[
\|f(x)\|_p^p\leq C\sum_{l} 2^{lp}\mu({\Omega}_l)^{p}\leq C\|{\widetilde S}(f)\|^p_p\leq C\|f\|^p_{H^p}
\]
since $\mu(\widetilde{\Omega}_l)\leq C\mu(\Omega_l).$

The proof of Theorem \ref{thm T1Hp} is concluded.
\end{proof}

\section{Criterion of the boundedness for singular integrals on Carleson measure and Campanato spaces}\label{sec:proof of theorem 1.12}

As mentioned in Section 1, we first prove the following the weak density argument.
\begin{lemma}\label{weaktopology}
 Suppose that $\frac{\omega}{\omega+\eta}<p\leq 1.$ Then $L^2(X) \cap \cmo^p(X)$ is dense in $\cmo^p(X)$ in the sense of the weak topology $(H^p(X),\cmo^p(X))$. Moreover precisely, for each $f \in \cmo^p(X)$ there exists a sequence $\left \{ f_n \right \} \in L^2(X) \cap \cmo^p(X)$ such that $\left \| f_n \right \|_{\cmo^p} \leq \left \| f \right \|_{\cmo^p}$ and moreover,
 \[
 \lim_{n \rightarrow \infty} \left \langle f_n,g \right \rangle = \left \langle f,g \right \rangle
 \]
 for all $g \in  H^p(X)$.
\end{lemma}
\begin{proof} Let $f$ be in $\cmo^p(X).$ By wavelet expansion, $f(x) = \sum_{k,\alpha} \left \langle \psi_\alpha^k ,f \right \rangle\psi_\alpha^k(x),$ we set
\begin{eqnarray*}
f_n(x):=\sum_{|k|\leq n}\  \sum_{Q_\alpha^k\subset B_n}\left \langle \psi_\alpha^k ,f \right \rangle
\psi_\alpha^k(x),
\end{eqnarray*}
where $B_n=\lbrace x: d(x_0, x)\leq n\rbrace$ and $x_0\in X$ is fixed.
It is easy to see that $f_n\in L^2(X).$ To see that $f_n\in \cmo^p(X),$ for any quasi dyadic cube $P\subset X$ in the sense of Auscher and Hyt\"onen,
\begin{eqnarray*}
&&\sup_{P}
 {1\over\mu(P)^{{2\over p} - 1}}
            \sum_{k'\in\mathbb{Z}, \alpha'\in\mathscr{Y}^{k'},
                 Q_{\alpha'}^{k'} \subset P}
            \big| \langle \psi_{\alpha'}^{k'}, f_n \rangle
            \big|^2\\
&=&\sup_{P}
 {1\over\mu(P)^{{2\over p} - 1}}
            \sum_{|k|\leq  n, Q_{\alpha'}^{k'} \subset P\cap B_n}
            \big| \langle \psi_{\alpha'}^{k'}, f \rangle\big|\leq \|f\|^2_{\cmo^p},
\end{eqnarray*}
which implies that $\|f_n\|_{\cmo^p}\leq\|f\|_{\cmo^p}.$

We now verify that $f_n$ converges to $f$ in the weak topology $(H^p(X), \cmo^p(X))$.
To do this, for any $h\in \GGp(\beta,\gamma)$, by the wavelet expansion,
$$
\langle f-f_n, h\rangle =  \langle f, \sum \limits_{|k|> n, or Q_\alpha^k \nsubseteq B_n}\psi_{\alpha}^k
            \langle \psi_{\alpha}^k, h \rangle .
$$
Note that $h$ is in both $H^p(X)$ and $\cmo^p(X)$ and moreover, $\sum \limits_{|k|> n, or Q_\alpha^k \nsubseteq B_n}\psi_{\alpha}^k(x)\langle \psi_{\alpha}^k, h \rangle$ tends to zero as $n$ tends to infinity in the $H^p(X)$ norm. Therefore, by the duality argument in Theorem C, $\langle f-f_n, h\rangle $ tends to 0 as $n$ tends to infinity. Further note that $ \GGp(\beta,\gamma)$ is dense in $H^p(X).$ Thus, for each $g\in H^p(X)$ and for any $\varepsilon>0$, there exists a function $h\in  \GGp(\beta,\gamma)$ such that $\|g-h\|_{H^p(X)}<\varepsilon$. Now by the duality and the fact that $\|f_n\|_{\cmo^p(X)}\leq \|f\|_{\cmo^p(X)}$, we have
\begin{eqnarray*}
\big|\langle f-f_n, g\rangle\big| &\leq & \big|\langle f-f_n, g-h\rangle\big|+ \big|\langle f-f_n, h\rangle\big|  \\[5pt]
&\leq&  \|f-f_n\|_{\cmo^p}\|g-h\|_{H^p}+ \big|\langle f-f_n, h\rangle\big|  \\[5pt]
&\leq&  2\varepsilon\|f\|_{\cmo^p}+ \big|\langle f-f_n, h\rangle\big|,
\end{eqnarray*}
which together with the fact that $\big|\langle f-f_n, h\rangle\big|$ tends to zero as $n$ tends to infinity implies that $\lim\limits_{n\rightarrow\infty}\langle f-f_n, g\rangle=0$. The proof of Lemma \ref{weaktopology} is completed.
\end{proof}

We now return to the proof of Theorem \ref{thm T1CMOP}. Suppose that $T$ is a singular integral with the kernel $K(x,y)$ satisfying the estimates \eqref{size of C-Z-S-I-O} and \eqref{x smooth of C-Z-S-I-O}, and $T$ is bounded on $L^2(X).$
To show the boundedness of $T$ on $\cmo^p(X),$ we first define the action of $T$ on $\cmo^p(X).$ To do this, given $f \in \cmo^p(X),$ by Lemma \ref{weaktopology}, there exists a sequence $\lbrace f_n\rbrace$ converges to $f$ in the weak topology $(H^p(X),\cmo^p(X)).$ We observe that for all $g \in H^p(X),$ $<Tf_n, g>$ has the limit as $n$ tends to infinity.
This is because that $T^*$ satisfies all conditions of Theorem \ref{thm T1Hp} and hence $T^*$ is bounded on $H^p(X).$ Thus, if $g$ is in the subspace $L^2(X)\cap H^p(X),$ then $T^*(g)\in L^2(X)\cap H^p(X),$ by the duality in Theorem C,
$$
\left | \left \langle Tf_n,g \right \rangle - \left \langle Tf_m,g \right \rangle \right |
= \left | \left \langle (f_n-f_m), T^*g \right \rangle \right |, $$
which tends to zero as $n,m$ tend to infinity and hence $<Tf_n, g>$ has the limit as $n$ tends to infinity.

Applying the density argument gives the observation. Note that the limit of $<Tf_n, g>$ as $n$ tends to infinity does not depend on the sequence $\lbrace f_n\rbrace$ satisfying Lemma \ref{weaktopology} and thus, we can define
\[
\left \langle Tf,g \right \rangle = \lim_{n \rightarrow \infty} \left \langle Tf_n,g \right \rangle
\]
for all $g \in H^p(X)$.

We claim that there exists a constant $C$ such that
\[
\|Tf\|_{\cmo^p}\leq C\|f\|_{\cmo^p}
\]
for all $f \in L^2(X)\cap \cmo^p(X)$.

Assuming the claim for the moment, by the definition of $Tf$ for $f\in \cmo^p(X),$ we have $\langle \psi_\alpha^k, Tf\rangle =\lim_{n \rightarrow \infty }\left \langle \psi_\alpha^k,Tf_n \right \rangle $ where $f_n$ is the sequence given by Lemma \ref{weaktopology}. This together with the claim implies that

$$
\left \|  Tf\right \|_{\cmo^p} \leq \lim_{n \rightarrow \infty} \inf \left \| Tf_n \right \|_{\cmo^p}$$
$$\leq C\lim_{n \rightarrow \infty} \inf \left \| f_n \right \|_{\cmo^p} \leq C \left \| f \right \|_{\cmo^p},$$
which completes the proof of the sufficiency of Theorem \ref{thm T1CMOP}.

We return to show the claim. The basic idea is to apply the duality in Theorem C. To be more precise, by Theorem C, for $f \in L^2(X) \cap \cmo^p(X)$ and $g \in L^2(X)\cap H^p(X)$,
\[\left | \left \langle Tf,g \right \rangle \right | = \left | \left \langle f,T^*g \right \rangle \right | \leq C \left \| f \right \|_{\cmo^p} \left \| T^*g \right \|_{H^p} \leq C \left \| f \right \|_{\cmo^p} \left \|  g\right \|_{H^p}.\]

This implies that for each $f \in L^2(X)\cap \cmo^p(X)$,  $\mathcal{L}_{f}(g) = \left \langle Tf,g \right \rangle$   defines a continuous linear functional on $L^2(X)\cap H^p(X)$.  Note that $L^2(X)\cap H^p(X)$ is dense in $H^p(X).$ Thus, $\mathcal{L}_f(g)=\langle T(f), g\rangle$ extends to a linear functional on $H^p(X)$ and the norm of this linear functional is dominated by $C\|f\|_{\cmo^p(X)}.$ By Theorem C, there exists $h \in \cmo^p(X)$ with $\|h\|_{\cmo^p}\leq C\|f\|_{\cmo^p}$ such that
$$\left \langle Tf,g \right \rangle = \mathcal{L}_{f}(g) = \left \langle h,g \right \rangle$$
for $g \in L^2(X) \cap H^p(X)$.

Particularly, note that all $\psi_\alpha^k(x) \in L^2(X) \cap H^p(X).$ Therefore,
\[\left \langle Tf,\psi_\alpha^k \right \rangle = \left \langle h,\psi_\alpha^k \right \rangle,\]
which, by the definition of space $\cmo^p(X),$ we obtain that
$$\left \| Tf \right \|_{\cmo^p} = \left \| h \right \|_{\cmo^p} \leq C \left \| f \right \|_{\cmo^p}.$$
The proof of the claim is concluded.

The necessary condition in Theorem \ref{thm T1CMOP} follows directly from the boundedness of $T$ on $\cmo^p(X)$ and the fact that the function 1 has the norm zero in $\cmo^p(X).$  The proof of Theorem \ref{thm T1CMOP} is concluded.

As a consequence of Theorem \ref{thm T1CMOP}, we have the following
\begin{corollary}
If $T$ satisfies the same condition as in Theorem \ref{thm T1CMOP}, then $T$ is bounded on Campanato space $\mathcal C_{\frac{1}{p}-1}(X)$ if and only if $T(1)=0.$
\end{corollary}
 The proof of this corollary follows directly from Theorem \ref{thm T1CMOP} and the following
\begin{prop} \label{cmopC}$\cmo^p(X)=\mathcal C_{\frac{1}{p}-1}(X), \frac{\omega}{\omega+\eta}<p\leq 1$ with equivalent norms.
\end{prop}
\begin{proof}
By Lemma \ref{lipcmop6}, we only need to show that there exists a constant $C$ such that for $\frac{\omega}{\omega+\eta}<p\leq 1,$
$$\|f\|_{\mathcal C_{\frac{1}{p}-1}}\leq C\|f\|_{\cmo^p}.$$
The basic idea to verify the above estimate is to apply the duality argument in Theorem C. To this end, for given $f\in \cmo^p(X), \frac{\omega}{\omega+\eta}<p\leq 1,$ by Theorem C, we can define a linear functional on $H^p(X)$ by $\mathcal{L}_f(g)=\langle f, g\rangle$ for $g\in H^p(X).$ Now fix a quasi-dyadic ball $Q$ and let $L^2_Q$ denote the space of all square integrable functions supported on $Q$. Let $L^2_{Q,0}$ denote its closed subspaces of functions with integral zero.  Note that each $g\in L^2_{Q,0}$ is a multiple of an $(p,2)$ atom for $H^p(X)$ and that $\|g\|_{H^p(X)}\leq \mu(Q)^{\frac{1}{p}-\frac{1}{2}}\|g\|_{L^2}$. Therefore, this linear functional $\mathcal L_f$ on $H^p(X)$ can extend to a linear functional on $L^2_{Q,0}$ with norm at most $C\mu(Q)^{\frac{1}{p}-\frac{1}{2}}\|f\|_{\cmo^p}$. By the Riesz representation theorem for Hilbert spaces $L^2_{Q,0}$, there exists an element $F^Q\in L^2_{Q,0}$ such that
$$\mathcal{L}_f(g)=\langle f, g\rangle=\int _Q F^Q(x)g(x)d\mu(x), \ \ \ \ \mathrm{ if} \ g\in L^2_{Q,0},$$
with
$$\|F^Q\|_{L^2_{Q}}=\Big\{\int _Q |F^Q(x)|^2d\mu(x)\Big\}^{\frac{1}{2}}\leq C\mu(Q)^{\frac{1}{p}-\frac{1}{2}}\|f\|_{\cmo^p}.$$
Thus $f$ must be a square integrable function on $Q$ and for each quasi ball $Q,$ we have such a function $F^Q$ such that on each ball $Q,$ $f$ differ from $F^Q$ by a constant. This implies that $f$ is a locally square integrable function on $X$ and on each quasi ball $Q$ there exists a constant $c_Q$ such that $f=F^Q+c_Q.$ Observe that
 \begin{eqnarray*}
 \bigg\{{1\over\mu(Q)^{{2\over p} -1}} \int_{Q}|f(x)-c_Q|^2d\mu(x)\bigg\}^{1/2}&=&\bigg\{{1\over\mu(Q)^{{2\over p} -1}} \int_{Q}|F^{Q}(x)|^2d\mu(x)\bigg\}^{1/2}\\
 &\leq& C {1\over\mu(Q)^{{1\over p} -\frac{1}{2}}}\mu(Q)^{\frac{1}{p}-\frac{1}{2}}\|f\|_{\cmo^p}\\
 &=&C\|f\|_{\cmo^p}.
 \end{eqnarray*}
Note that on each quasi ball $Q$ and for any constant $c, \int_Q|f(x)-f_Q|^2d\mu(x)\leq C\int_Q|f(x)-c|^2d\mu(x).$ The above estimate yields that
$$\bigg\{{1\over\mu(Q)^{{2\over p} -1}} \int_{Q}|f(x)-f_Q|^2d\mu(x)\bigg\}^{1/2}\leq C\|f\|_{\cmo^p},$$
which implies that $f\in {\mathcal L}_{\frac{1}{p}-1}(X)$ and
$$\|f\|_{{\mathcal L}_{\frac{1}{p}-1}(X)}\leq C\|f\|_{\cmo^p}.$$
\end{proof}


\bigskip

\noindent School  of Mathematic Sciences, South China Normal University, Guangzhou, 510631, P.R. China.

\noindent {\it E-mail address}:
\texttt{20051017@m.scnu.edu.cn}

\medskip

\noindent Department of Mathematics, Auburn University, AL
36849-5310, USA.

\noindent {\it E-mail address}: \texttt{hanyong@auburn.edu}

\medskip

\noindent Department of Mathematics, Macquarie University, NSW, 2109, Australia.

\noindent {\it E-mail address}: \texttt{ji.li@mq.edu.au}

\end{document}